\documentclass[12pt, reqno, a4paper]{amsart}
\usepackage[all]{xy}

\usepackage{ amssymb, amsmath, enumerate, amsfonts, amsthm, mathrsfs, url, bm, mathtools, comment}

\numberwithin{equation}{section}

\usepackage{xcolor}  	
\usepackage[backref=page]{hyperref}
\hypersetup{
	colorlinks,
    linkcolor={blue!60!black},
    citecolor={blue!60!black},
    urlcolor={red!60!black}
}

\usepackage{color}

\usepackage[margin=1.25in]{geometry}

\RequirePackage{doi}

\usepackage[square,sort,comma,numbers]{natbib}
\makeatletter
\@namedef{subjclassname@2020}{%
  \textup{2020} Mathematics Subject Classification}
\makeatother

\newtheorem{theorem}{Theorem}[section]
\newtheorem{question}[theorem]{Question}

\newtheorem{corollary}[theorem]{Corollary}

\newtheorem{proposition}[theorem]{Proposition}
\newtheorem{lemma}[theorem]{Lemma}
\newtheorem{example}[theorem]{Example}

\theoremstyle{definition}
\newtheorem{definition}[theorem]{Definition}
\newtheorem*{remark}{Remark}

\numberwithin{theorem}{section}

\author{Manuel Saavedra}
\address{Departamento de Matematica, Universidade Estadual Paulista, Sao José
	do Rio Preto, SP, Brasil}
\email{manuel.saavmath@gmail.com}
\thanks{The author was partially supported by Fapesp Grant 2026/08439-1}

\begin{document}

\title{Observation Schemes and Irregularity in Linear Dynamics}

\dedicatory{A la memoria de mi mamá Elva y mi hermanita Leyla}

\begin{abstract}	We develop a structural framework for irregularity in linear dynamics centered on the space $\Upsilon$ of observation schemes. This approach separates the underlying dynamical behavior from the observation mechanism and provides a unified setting in which classical notions such as Li--Yorke chaos, mean Li--Yorke chaos, and distributional chaos arise as particular cases corresponding to Dirac measures and Cesàro averages. We establish two abstract criteria ensuring the existence of large linear structures, leading to dense-lineability and spaceability results for both absolutely $(\mu_m)$-irregular and distributionally $(\mu_m)$-irregular vectors. Furthermore, under a natural density assumption, we obtain a trichotomy describing the global behavior of irregularity across $\Upsilon$, together with rigidity phenomena for the classes of observation schemes generating each type of chaotic behavior.
\end{abstract}

\maketitle

\section{Introduction}

Li--Yorke chaos was introduced in \cite{LiYor75} for interval maps and was subsequently adapted to the linear setting. A fundamental result of Bermúdez et al.~\cite{BeBoMaPe} asserts that Li--Yorke chaos for a bounded operator on a Banach space is equivalent to the existence of an \emph{irregular vector}, a notion previously introduced by Beauzamy~\cite{Beauzamy88}. Recall that a vector $x\in X$ is said to be \emph{irregular} for an operator $T\in\mathcal{L}(X)$ if
\[
\liminf_{n\to\infty}\Vert{T^n x\Vert}=0
\qquad\text{and}\qquad
\limsup_{n\to\infty}\Vert{T^n x\Vert}=\infty.
\]

Bernardes et al.~\cite{BeBoPe20} later studied mean Li--Yorke chaos for bounded operators on Banach spaces and proved that it is equivalent to the existence of a \emph{mean irregular vector}, namely, a vector $x\in X$ satisfying
\[
\liminf_{n\to\infty}
\frac{1}{n}\sum_{k=1}^{n}\Vert{T^k x\Vert}=0
\qquad\text{and}\qquad
\limsup_{n\to\infty}
\frac{1}{n}\sum_{k=1}^{n}\Vert{T^k x\Vert}=\infty.
\]

Distributional chaos was introduced by Schweizer and Smital in \cite{ScSm}. In the setting of linear dynamics, Bernardes et al.~\cite{BeBoMuPe13} proved that a continuous linear operator on a Fréchet space is distributionally chaotic if and only if it admits a \emph{distributionally irregular vector}. More precisely, a vector $x\in X$ is distributionally irregular for $T\in\mathcal{L}(X)$ if there exist $\beta\in\mathbb{N}$ and sets $A,B\subset\mathbb{N}$ with
\[
\overline{\text{dens}}(A)=\overline{\text{dens}}(B)=1
\]
such that
\[
\lim_{\substack{n\to\infty, n\in A}}T^n x=0
\qquad\text{and}\qquad
\lim_{\substack{n\to\infty, n\in B}}\Vert{T^n x\Vert}_{\beta}=\infty.
\]
Existing notions of chaos in linear dynamics depend implicitly on the underlying observation scheme. In practice, the literature has predominantly focused on two schemes: Dirac measures and Cesàro means (see \eqref{lenses}). Rather than introducing new notions of chaos, we adopt a structural viewpoint in which the underlying dynamical behavior and the observation scheme are treated as independent components. This leads to the study of chaos through the space $\Upsilon$ of observation schemes, which becomes the primary object governing the emergence of irregularity. This viewpoint allows one to treat irregularity as a global phenomenon across observation schemes, revealing structural properties such as genericity and rigidity that are not accessible when working with a fixed scheme.

More precisely, $\Upsilon$ consists of sequences $(\mu_m)_m$ of probability measures on $\mathbb{N}$ that asymptotically vanish on finite sets. This condition prevents concentration of mass on bounded time intervals and enforces a genuinely asymptotic regime of observation. A precise definition and its topological properties are given in Section~\ref{b.t.o.s}.

The classical observation schemes arise naturally within this framework. Indeed, for $T\in\mathcal{L}(X)$, one has
\begin{equation}\label{lenses}
	\int_{\mathbb{N}} \|T^{t}x\| \, d\mu_{m}(t) =
	\begin{cases}
		\|T^{m}x\| & \text{if } \mu_{m} = \delta_{m}, \\[6pt]
		\displaystyle{\frac{1}{m}\sum_{k=1}^{m}\|T^{k}x\|} & \text{if } \mu_{m} = \displaystyle{\frac{1}{m}\sum_{k=1}^{m}\delta_{k}}.
	\end{cases}
\end{equation}
Thus, classical notions such as Li--Yorke chaos and mean Li--Yorke chaos arise as particular instances of the same observation-driven framework.

Our main result is a trichotomy that describes the global behavior of irregularity across the space of observation schemes $\Upsilon$.

\begin{theorem}[Trichotomy for Observation Schemes (Banach-space case)]
	Let $X$ be a separable real or complex Banach space and let 
	$T \in \mathcal{L}(X)$. Assume that the set 
	\[
	\{x \in X : T^{n}x \to 0\}
	\]
	is dense in $X$. Then exactly one of the following alternatives holds:
	\begin{enumerate}
		\item[(I)] $T^{n}x \to 0$ for every $x \in X$;
		
		\item[(II)] for every $(\mu_m)\in\Upsilon$, the operator $T$ admits both a dense absolutely $(\mu_m)$-irregular manifold and a dense distributionally $(\mu_m)$-irregular manifold;
		
		\item[(III)] there exists a residual set $\Upsilon_0\subset\Upsilon$ such that, for every $(\mu_m)\in\Upsilon_0$, the operator $T$ admits a dense absolutely $(\mu_m)$-irregular manifold but no distributionally $(\mu_m)$-irregular vector.
	\end{enumerate}
\end{theorem}

In particular, alternative~(III) provides a separation between absolute \((\mu_m)\)-irregularity and distributional $(\mu_m)$-irregularity for a residual class of observation schemes. This provides a partial answer to \cite[Question 16]{BeBoPe20}, showing that, although the classical problem for Cesàro averages remains open, such a separation occurs generically at the level of observation schemes.

Under the above assumptions, the irregularity alternative can be substantially strengthened. 
Indeed, if there exists a decreasing sequence $(E_n)_{n\in\mathbb{N}}$ of infinite-dimensional closed subspaces of $X$ such that
\[
\sup_{n\in\mathbb N}\|T^{n}|_{E_{n}}\|<\infty,
\]
then there exists a residual subset of lenses in $\Upsilon$ for which, for every $(\mu_m)_m$ in this set, the operator \(T\) admits both a dense absolutely \((\mu_m)\)-irregular manifold and a closed infinite-dimensional absolutely \((\mu_m)\)-irregular manifold. In particular, this provides new instances of spaceability phenomena within the framework of linear chaos, highlighting the richness of irregular behavior under suitable structural assumptions.

Prior to the trichotomy, we establish that chaos satisfies rigid topological laws when viewed through the space $\Upsilon$. More precisely, for a fixed operator \(T\), the class of observation schemes for which \(T\) is densely \((\mu_m)\)-distributionally chaotic is either maximal or topologically negligible in \(\Upsilon\) (see Theorem~\ref{dicho-dist}). In sharp contrast, the class of schemes for which \(T\) is densely \((\mu_m)\)--Li--Yorke chaotic always forms a \(G_{\delta}\)-subset of \(\Upsilon\) (see Theorem~\ref{G_delta-dense-LY}).

The paper is organized as follows. Section~3 introduces absolute $(\mu_m)$-boundedness and establishes its basic properties. Section~4 develops a general framework for $(\mu_m)$-Li--Yorke chaos, while Section~5 studies the associated $\mathrm{D}$-phenomenon by means of Furstenberg--Borel families (see Definition~\ref{Furs-Borel}). Section~6 provides abstract criteria for the existence of large linear structures. These criteria are applied in Sections~7 and~8 to obtain dense-lineability and spaceability results for absolutely $(\mu_m)$-irregular and distributionally $(\mu_m)$-irregular vectors. Finally, Section~9 is devoted to the proof of the trichotomy for observation schemes.


\section*{Acknowledgements}

The author would like to thank Alexander Arbieto, Nilson Bernardes, Ali Messaoudi and Manuel Stadlbauer for their valuable comments and insightful discussions related to this work.

\section{Setting and Notation}

Throughout this work, we consider a metric space $\Lambda$  and
\((T_t)_{t\in\Lambda}\) a locally equicontinuous family of continuous linear
operators from an \(F\)-space \(X\) into a normed space or, more generally,
into a metrizable locally convex space \(Y\).
The topology of \(Y\) is assumed to be generated by a countable directed
family of seminorms \((\|\cdot\|_{\ell})_{\ell\in\mathbb{N}}\), and the
associated translation-invariant metric on \(Y\) is given by
\[
\rho(x,y)
:=\sum_{\ell=1}^{\infty}2^{-\ell}\,\frac{\|x-y\|_{\ell}}{1+\|x-y\|_{\ell}},
\qquad x,y\in Y.
\]

To fix notation, we denote by \(\mathrm{D}\) a translation-invariant metric on
\(X\) that induces the topology of \(X\).
In the particular case \(X=Y\), that is, when
\((T_t)_{t\in\Lambda}\subset \mathcal{L}(X)\), we additionally assume that
\(X\) is an infinite-dimensional Banach space or a Fr\'echet space, endowed
with its norm \(\|\cdot\|\) and metric \(\rho\), respectively.

On the other hand, \((\mu_i)_{i\in I}\) denotes a family of Borel probability
measures on \(\Lambda\), each with compact support (unless stated otherwise), where
\(I\subset \mathbb{R}_{\ge 0}\) is unbounded and such that, for every
\(i\in I\),
\begin{equation}\label{compact-soporte}
	\overline{\bigcup_{\substack{j\in I\\ j\le i}}
		\operatorname{supp}(\mu_j)}
\end{equation}
is a compact subset of \(\Lambda\).
Moreover, for every compact set \(K\subset \Lambda\),
\[
\lim_{i\to\infty}\mu_i(K)=0.
\]

Further developments and related results can be found in \cite{BeBoMuPe13, BeBo15, BeBoPeWu18, BoKo19, JiLi25}. For background on linear dynamics and related notions such as hypercyclicity, we refer to \cite{BaMa, GrPe}.


\section{Absolute $(\mu_i)_{i\in I}$-boundedness}

In this section, we introduce the notion of absolute
$(\mu_i)_{i\in I}$-boundedness and establish its basic properties. In
particular, we obtain a characterization and a useful consequence that
will be applied later to the study of Li--Yorke-type chaos.

The most classical boundedness condition in this setting is power
boundedness. If $X$ is a Banach space, an operator
$T\in\mathcal{L}(X)$ is said to be \emph{power bounded} if
\[
\sup_{n\in\mathbb{N}}\|T^n\|<\infty.
\]
If $X$ is a Fr\'echet space, $T\in\mathcal{L}(X)$ is said to be power
bounded when the family
\[
\{T^n:n\in\mathbb{N}\}
\]
is equicontinuous. Equivalently, every orbit
\[
\mathcal{O}(x):=\{T^n x:n\in\mathbb{N}\},
\qquad x\in X,
\]
is bounded in $X$.

Luo and Huo~\cite{LuHo15} introduced the notion of absolute Ces\`aro
boundedness for operators on Banach spaces. This property is closely
related to several other boundedness conditions, including mean square
boundedness, absolute Kreiss boundedness, and absolute strong Kreiss
boundedness; see, for instance,
\cite{BeBoMuPe20,BeBoPe20,CoCuEiLi}.

\begin{definition}
	Let $X$ be a Banach space and let $T\in\mathcal{L}(X)$. We say that
	$T$ is \emph{absolutely Ces\`aro bounded} if there exists $C>0$ such
	that
	\[
	\sup_{N\geq 1}
	\frac{1}{N}\sum_{k=1}^{N}\|T^k x\|
	\leq C\|x\|
	\qquad\text{for every }x\in X.
	\]
\end{definition}

These boundedness conditions admit a common formulation in terms of
averages with respect to probability measures. For example, absolute
Kreiss boundedness is associated with the geometric probability measures
$\mu_r$ on $\mathbb{N}\cup\{0\}$ defined by
\begin{equation}\label{geo}
	\mu_r(\{n\})=(1-r)r^n,
	\qquad n\geq 0,\quad 0<r<1.
\end{equation}
For every $x\in X$, one then has
\[
\int_{\mathbb{N}\cup\{0\}}\|T^t x\|\,d\mu_r(t)
=
(1-r)\sum_{n=0}^{\infty}r^n\|T^n x\|.
\]
Similarly, absolute strong Kreiss boundedness is associated with the
Poisson probability measures
\begin{equation}\label{poisson}
	\mu_r(\{n\})=e^{-r}\frac{r^n}{n!},
	\qquad n\geq 0,\quad r>0,
\end{equation}
for which
\[
\int_{\mathbb{N}\cup\{0\}}\|T^t x\|\,d\mu_r(t)
=
e^{-r}\sum_{n=0}^{\infty}\frac{r^n}{n!}\|T^n x\|.
\]

These observations lead to the following general notion of boundedness. Note that the probability measures considered in the previous examples need not have compact support. Accordingly, in the following definition we do not impose any compactness assumption on the family $(\mu_i)_{i\in I}$.

\begin{definition}\label{def:abs-bdd}
	Let $(\mu_i)_{i\in I}$ be a family of probability measures on $\Lambda$
	(not necessarily compactly supported and a general index set \(I\)). A family $(T_t)_{t\in\Lambda}\subset\mathcal{L}(X,Y)$ is said to be
	\emph{\(\beta\)-absolutely $(\mu_i)_{i\in I}$-bounded}
	if there exist constants $C>0$ and $\delta>0$ such that
	\begin{equation}\label{abs-measure}
		\sup_{\substack{x\in X\\ \mathrm{D}(x,0)<\delta}}
		\ \sup_{i\in I}
		\int_{\Lambda}\|T_t x\|_{\beta}\, d\mu_i(t)
		\le C.
	\end{equation}
	We say that $(T_t)_{t\in\Lambda}$ is \emph{absolutely $(\mu_i)_{i\in I}$-bounded}
	if, for every $\beta\in\mathbb{N}$, the family $(T_t)_{t\in\Lambda}$ is \(\beta\)-absolutely
	$(\mu_i)_{i\in I}$-bounded.
\end{definition}

In the Banach-space setting, where $X=Y$ and
$(T_t)_{t\in\Lambda}\subset\mathcal{L}(X)$, the local condition
\eqref{abs-measure} is equivalent, by homogeneity, to the existence of a
constant $C>0$ such that
\[
\sup_{i\in I}
\int_{\Lambda}\|T_t x\|\,d\mu_i(t)
\leq C\|x\|
\qquad\text{for every }x\in X.
\]

Consequently, given $T\in\mathcal{L}(X)$, applying
Definition~\ref{def:abs-bdd} to the family of iterates
\[
(T^n)_{n\in\mathbb{N}\cup\{0\}}
\]
recovers the classical Kreiss-type boundedness conditions. More
precisely, $T$ is \emph{absolutely Kreiss bounded} if
$(T^n)_{n\in\mathbb{N}\cup\{0\}}$ is absolutely
$(\mu_r)_{0<r<1}$-bounded with respect to the geometric probability
measures defined in \eqref{geo}. Similarly, $T$ is
\emph{absolutely strong Kreiss bounded} if
$(T^n)_{n\in\mathbb{N}\cup\{0\}}$ is absolutely
$(\mu_r)_{r>0}$-bounded with respect to the Poisson probability measures
defined in \eqref{poisson}.

It is worth noting that Cohen et al.\ \cite[Proposition~3.5]{CoCuEiLi} proved that, for bounded linear operators on Banach spaces, absolute Ces\`aro boundedness is equivalent to absolute Kreiss boundedness.

\begin{theorem}\label{bounded-beta}
	Let $(T_t)_{t\in \Lambda}\subset \mathcal{L}(X,Y)$ and let $\beta\in\mathbb{N}$. Then the following assertions are equivalent:
	\begin{enumerate}
		\item $(T_t)_{t\in \Lambda}$ is \(\beta\)-absolutely $(\mu_i)$-bounded.
		\item For every $x\in X$,
		\[
		\sup_{i\in I}\int_{\Lambda} \|T_t x\|_{\beta}\, d\mu_i(t)<\infty.
		\]
		\item  $(T_t)_{t \in \Lambda}$ satisfies
		\[
		\lim_{r \downarrow 0^{+}} \sup_{\substack{x\in X\\ \mathrm{D}(x,0)<r}} \sup_{i \in I} \int_\Lambda \|T_t x\|_\beta \, d\mu_i(t) = 0.
		\]
	\end{enumerate}
\end{theorem}

\begin{remark}
	In the proof of Theorem~\ref{bounded-beta}, the compact-support
	assumption on the measures $(\mu_i)_{i\in I}$ is used only in the
	implication \emph{(2)}$\Rightarrow$\emph{(1)}. The remaining
	implications hold for arbitrary families of probability measures and
	for an arbitrary index set $I$. In particular, no additional structure
	on $I$ is required.
	
	Moreover, if $\Lambda$ is a $\sigma$-compact metric space, then the
	compact-support assumption may be omitted. Indeed, let
	$(K_n)_{n\geq 1}$ be an increasing sequence of compact subsets of
	$\Lambda$ such that $\Lambda=\bigcup_{n\geq 1}K_n$. For each $i\in I$,
	choose $n(i)\in\mathbb{N}$ such that $\mu_i(K_{n(i)})\geq 1/2$, and,
	for $j\geq n(i)$, define
	\[
	\mu_{i,j}(A)
	:=
	\frac{\mu_i(A\cap K_j)}{\mu_i(K_j)}.
	\]
	Then $\mu_{i,j}$ is compactly supported and, for every $x\in X$,
	\[
	\int_{\Lambda}\|T_t x\|_{\beta}\,d\mu_{i,j}(t)
	\leq
	2\int_{\Lambda}\|T_t x\|_{\beta}\,d\mu_i(t),
	\]
	while
	\[
	\int_{\Lambda}\|T_t x\|_{\beta}\,d\mu_{i,j}(t)
	\longrightarrow
	\int_{\Lambda}\|T_t x\|_{\beta}\,d\mu_i(t)
	\qquad (j\to\infty).
	\]
	Applying Theorem~\ref{bounded-beta} to the family
	$\bigl(\mu_{i,j}\bigr)_{i\in I,\,j\geq n(i)}$ and then letting
	$j\to\infty$ yields the conclusion.
\end{remark}

\begin{proof}[Proof of Theorem \ref{bounded-beta}]
	$(1)\Rightarrow(2)$.  
	Assume that $(T_t)_t$ is \(\beta\)-absolutely $(\mu_i)$-bounded. Then there exist constants
	$C>0$ and $\delta>0$ such that
	\[
	\sup_{i\in I}\int_{\Lambda} \|T_t x\|_{\beta}\, d\mu_i(t)\leq C,
	\qquad \text{whenever } \mathrm{D}(x,0)<\delta.
	\]
	Fix $x\in X$. Choose $s>0$ such that $\mathrm{D}(x/s,0)<\delta$. Thus,
	\[
	\sup_{i\in I}\int_{\Lambda} \|T_t x\|_{\beta}\, d\mu_i(t)
	\leq Cs < \infty.
	\]
	This proves $(2)$.
	
	\medskip
	\noindent
	$(2)\Rightarrow(1)$.  
	Assume that $(2)$ holds. For each $n\in\mathbb{N}$, define
	\[
	A_n:=\Bigl\{x\in X : \sup_{i\in I}\int_{\Lambda} \|T_t x\|_{\beta}\, d\mu_i(t)\leq n\Bigr\}.
	\]
	We first show that each $A_n$ is closed in $X$. Let $(x_\ell)_\ell \subset A_n$ be a sequence converging to some $x \in X$. Fix $\varepsilon > 0$ and $i \in I$. Since $\mathrm{supp}(\mu_i)$ is compact, the family $(T_t)_{t \in \mathrm{supp}(\mu_i)}$ is equicontinuous. Consequently, there exists $r > 0$ such that
	\[
	T_t\bigl(\{z\in X:\mathrm{D}(z,0)<r\}\bigr)\subset
	\{y\in Y:\|y\|_{\beta}<\varepsilon\}
	\quad\text{for all } t\in\mathrm{supp}(\mu_i).
	\]
	Choose $\ell\in\mathbb{N}$ such that $\mathrm{D}(x_\ell-x,0)<r$. Then
	\begin{align*}
		\int_{\Lambda} \|T_t x\|_{\beta}\,d\mu_i(t)
		&\le
		\int_{\Lambda} \|T_t x_\ell\|_{\beta}\,d\mu_i(t)
		+
		\int_{\Lambda} \|T_t(x_\ell-x)\|_{\beta}\,d\mu_i(t)\\
		&\le n+\varepsilon.
	\end{align*}
	Since $\varepsilon>0$ and $i\in I$ are arbitrary, it follows that $x\in A_n$. Hence $A_n$ is closed in $X$.
	
	By assumption,
	\[
	X=\bigcup_{n\in\mathbb{N}} A_n.
	\]
	By the Baire Category Theorem, there exist $k\in\mathbb{N}$ and a nonempty open set
	$U\subset X$ such that $U\subset A_k$. Then $U-U$ is an open neighborhood of $0$ and
	\[
	U-U \subset \Bigl\{x\in X : \sup_{i\in I}\int_{\Lambda} \|T_t x\|_{\beta}\, d\mu_i(t)\leq 2k\Bigr\}.
	\]
	Hence, there exists $\delta>0$ such that
	\[
	\{x\in X : \mathrm{D}(x,0)<\delta\}\subset U-U,
	\]
	and consequently,
	\[
	\sup_{\substack{x\in X\\ \mathrm{D}(x,0)<\delta}}
	\sup_{i\in I}\int_{\Lambda} \|T_t x\|_{\beta}\, d\mu_i(t)\leq 2k.
	\]
	This shows that $(T_t)_t$ is \(\beta\)-absolutely $(\mu_i)_{i\in I}$-bounded.
	
	The implication (3) $\Rightarrow$ (1) is immediate. Now assume that (1) holds; that is, there exist $C>0$ and $\delta>0$ such that
	\[
	g(\delta):=\sup_{\substack{x\in X\\ \mathrm{D}(x,0)<\delta}}	\sup_{i \in I} \int_{\Lambda} \|T_t x\|_{\beta}\, d\mu_i(t)\leq C.
	\] 
	If (3) does not hold and since $g$ is increasing, then there exists $a>0$ such that 
	\[
	\lim_{r\downarrow 0^{+}} \sup_{\substack{x\in X\\ \mathrm{D}(x,0)<r}} \sup_{i \in I} \int_{\Lambda} \|T_t x\|_\beta \, d\mu_i(t)=a.
	\]
	We can choose a sequence $y_{n}\rightarrow 0$ in $X$ such that $\displaystyle{\sup_{i \in I}} \int_{\Lambda} \|T_t y_{n}\|_{\beta}\, d\mu_i(t)\geq \frac{a}{2}$ and $\mathrm{D}(n y_{n},0)<\delta$. Hence, for every $n\in \mathbb{N}$ we have
	\[
	\frac{a n}{2}\leq \sup_{i \in I} \int_{\Lambda} \|T_t (n y_{n})\|_{\beta}\, d\mu_i(t)\leq \sup_{\substack{x\in X\\ \mathrm{D}(x,0)<\delta}}\sup_{i \in I} \int_{\Lambda} \|T_t x\|_{\beta}\, d\mu_i(t) \leq C,
	\]
	which yields a contradiction. Thus $a=0$, and (1) implies (3).
\end{proof}

\begin{corollary}\label{cor:power-bounded} Let $X$ be a Fréchet space and let $T\in\mathcal{L}(X)$. Then $T$ is power bounded if and only if it is absolutely $(\delta_m)_{m\in\mathbb{N}}$-bounded.
\end{corollary}

The preceding theorem leads to the following statement, which will play a central role in our analysis of Li--Yorke type chaos. In the particular discrete setting where $T_t := T^t$ for $t \in \mathbb{N}$, with $T$ a bounded linear operator on a Banach space, and
\[
\mu_m := \frac{1}{m}\sum_{k=1}^{m}\delta_k,
\]
a related result is obtained in \cite[Theorem~4]{BeBoPe20}.

\begin{corollary}\label{not-abs-bounded}
	Let $(T_t)_{t \in \Lambda} \subset \mathcal{L}(X, Y)$. For $\beta\in \mathbb{N}$, the following assertions are equivalent:
	\begin{enumerate}
		\item $(T_t)_{t\in \Lambda}$ is not \(\beta\)-absolutely $(\mu_i)$-bounded.
		\item There exists $x\in X$ such that
		\[
		\sup_{i\in I}\int_{\Lambda} \|T_t x\|_{\beta}\, d\mu_i(t)=\infty.
		\]
		\item The set
		\[
		\Bigl\{ x \in X : \sup_{i \in I} \int_\Lambda \|T_t x\|_{\beta} \, d\mu_i(t) = \infty \Bigr\}
		\]
		is residual in $X$.
	\end{enumerate}
\end{corollary}

\begin{proof}
	The equivalence between $(1)$ and $(2)$ follows from the previous theorem, and the implication $(3)\Rightarrow(2)$ is immediate.
	
	To prove $(2)\Rightarrow(3)$, fix $q \in X$ satisfying $(2)$ and define
	\[
	Z := \Bigl\{ y \in X : \sup_{i \in I} \int_{\Lambda} \|T_t y\|_{\beta} \, d\mu_i(t) = \infty \Bigr\}.
	\]
	For any $y \in X \setminus Z$ and any $\lambda \neq 0$, we have $y+\lambda q \in Z$, which implies
	$y \in \overline{Z}$. Hence, $Z$ is dense in $X$. Since $Z$ is clearly a $G_\delta$ subset of $X$,
	it follows that $Z$ is residual in $X$.
\end{proof}



\section{A framework for $(\mu_i)$-Li--Yorke chaos}

Let $(Z,\vartheta)$ be a metric space and let $f\colon Z\to Z$ be
continuous. The map $f$ is said to be \emph{Li--Yorke chaotic} if there
exists an uncountable set $\Gamma\subset Z$ such that, for every pair of
distinct points $x,y\in\Gamma$,
\begin{equation}\label{LY-clasico}
	\liminf_{n\to\infty}
	\vartheta\bigl(f^n(x),f^n(y)\bigr)=0
	\quad\text{and}\quad
	\limsup_{n\to\infty}
	\vartheta\bigl(f^n(x),f^n(y)\bigr)>0.
\end{equation}

As recalled in the Introduction, Berm\'udez et al.~\cite{BeBoMaPe}
proved that, for bounded operators on Banach spaces, Li--Yorke chaos is
equivalent to the existence of an irregular vector. Bernardes et
al.~\cite{BeBo15} subsequently extended this characterization to
continuous linear operators on Fr\'echet spaces. In this setting, a
vector $x\in X$ is irregular for $T$ if there exists
$\beta\in\mathbb{N}$ such that
\[
\liminf_{n\to\infty}\rho(T^n x,0)=0
\qquad\text{and}\qquad
\limsup_{n\to\infty}\|T^n x\|_{\beta}=\infty.
\]
The existence of such a vector is again equivalent to Li--Yorke chaos.

Bernardes et al.~\cite{BeBoPe20} introduced mean Li--Yorke chaos for
bounded operators on Banach spaces. The corresponding notion for
continuous linear operators on Fr\'echet spaces was later studied by
Jiang and Li~\cite{JiLi25}.

More recently, Arbieto and Saavedra~\cite{ArSa} developed a notion of
Li--Yorke chaos for sequences of continuous linear operators from an
$F$-space into a normed space, formulated directly in terms of irregular
vectors. This approach places irregularity
at the center of the theory.

Following this point of view, we take irregularity as the basic notion
underlying Li--Yorke-type chaos in linear dynamics. This is also
consistent with the role played by dense irregular manifolds and dense
mean irregular manifolds, which provide particularly strong linear
manifestations of Li--Yorke chaotic behavior.

\begin{definition}\label{def-irre}
	A vector $x \in X$ is said to be \emph{\(\beta\)-absolutely $(\mu_i)_{i \in I}$-irregular}
	for the family $(T_t)_{t \in \Lambda}$ if
	\[
	\liminf_{\substack{i\rightarrow \infty\\ i\in I}}
	\int_{\Lambda} \rho\!\left(T_t x,0\right)\, d\mu_i(t) = 0,\quad \text{and}
	\quad
	\limsup_{\substack{i\rightarrow \infty\\ i\in I}}
	\int_{\Lambda} \|T_t x\|_{\beta}\, d\mu_i(t) = \infty .
	\]
	
	The family $(T_t)_{t \in \Lambda}\subset \mathcal{L}(X,Y)$ is said to be
	\emph{$(\mu_i)_{i \in I}$-Li--Yorke chaotic} if there exist an uncountable set
	$\Gamma \subset X$ and a seminorm $\|\cdot\|_{\beta}$ such that $x - y$ is
	\(\beta\)-absolutely $(\mu_i)_{i \in I}$-irregular
	for every pair of distinct points $x,y \in \Gamma$.
	
	It is said to be \emph{densely $(\mu_i)_{i \in I}$-Li--Yorke chaotic} if, in addition,
	$\Gamma$ can be chosen to be dense in $X$.
\end{definition}

It is worth noting that the standing assumptions—namely, that the measures $\mu_i$ are Borel probability measures with compact support, that the family $(T_t)_{t\in\Lambda}$ is locally equicontinuous and the condition (\ref{compact-soporte})—are essential to guarantee the appropriate topological structure of the sets under consideration. In particular, they ensure that
\[
\left\{ x  : \liminf_{i\rightarrow \infty} \int_{\Lambda} \rho(T_t x,0) \, d\mu_i = 0 \right\}
= \bigcap_{m \in \mathbb{N}} \bigcup_{\substack{i>m\\ i\in I}}
\left\{ x : \int_{\Lambda} \rho(T_t x,0) \, d\mu_i < \frac{1}{m} \right\},
\]
and
\[
\left\{ x : \limsup_{i\rightarrow \infty} \int_{\Lambda} \|T_t x\|_{\beta} \, d\mu_i= \infty \right\}
= \bigcap_{m \in \mathbb{N}} \bigcup_{\substack{i>m\\ i\in I}}
\left\{ x : \int_{\Lambda} \|T_t x\|_{\beta} \, d\mu_i > m \right\},
\]
are $G_{\delta}$-subsets of $X$.

\begin{proposition}\label{LY-irr}
	Let $(T_t)_{t} \subset \mathcal{L}(X)$ be a family of continuous linear operators.
	The following assertions are equivalent:
	\begin{enumerate}
		\item $(T_t)_{t}$ is $(\mu_i)$-Li--Yorke chaotic;
		\item $(T_t)_{t}$ admits a \(\beta\)-absolutely $(\mu_i)$-irregular vector
		for some $\beta \in \mathbb{N}$.
	\end{enumerate}
\end{proposition}

\begin{proof}
	The implication $(1)\Rightarrow(2)$ follows directly from the definition.
	Conversely, we argue by a classical construction. Assume that there exists
	$q\in X$ which is $\beta$-absolutely $(\mu_i)$-irregular for some
	$\beta\in\mathbb{N}$. Consider the uncountable set \(\Gamma := \mathbb{K}\,q \subset X \).
	
	Let $a,b\in\Gamma$ with $a\neq b$. Then there exists
	$\lambda\in\mathbb{K}\setminus\{0\}$ such that $a-b=\lambda q$. Hence
	\[
	\liminf_{i\to\infty}
	\int \rho(T_t a, T_t b)\, d\mu_i(t)
	\le (|\lambda|+1)
	\liminf_{i\to\infty}
	\int \rho(T_t q,0)\, d\mu_i(t)
	=0,
	\]
	and
	\[
	\limsup_{i\to\infty}
	\int \|T_t(a-b)\|_{\beta}\, d\mu_i(t)
	=
	|\lambda|
	\limsup_{i\to\infty}
	\int \|T_t q\|_{\beta}\, d\mu_i(t)
	=\infty.
	\]
	Therefore $a-b$ is again $\beta$-absolutely $(\mu_i)$-irregular.
	Consequently, $(T_t)_t$ is $(\mu_i)$-Li--Yorke chaotic.
\end{proof}

It is immediate to observe that if the family $(T_t)_t$ is $(\mu_i)$-Li--Yorke chaotic, then $(T_t)_t$ cannot be absolutely $(\mu_i)$-bounded.

Recall that an operator $T$ on a Banach space $X$ is said to be
\emph{mean Li--Yorke chaotic} \cite[Definition~1]{BeBoPe20} if there exists
an uncountable subset $\Gamma \subset X$ such that every pair $(x,y)$ of
distinct points in $\Gamma$ is a mean Li--Yorke pair for $T$, that is,
\[
\liminf_{N\to\infty} \frac{1}{N}\sum_{k=1}^{N}\|T^{k}(x-y)\| = 0
\quad \text{and} \quad
\limsup_{N\to\infty} \frac{1}{N}\sum_{k=1}^{N}\|T^{k}(x-y)\| > 0 .
\]
If $\Gamma$ can be chosen dense in $X$, then $T$ is said to be
\emph{densely mean Li--Yorke chaotic}.

\begin{proposition}
	Let $X$ be a separable infinite-dimensional Banach space and let $T\in\mathcal{L}(X)$. Then:
	\begin{enumerate}
		\item $T$ is Li--Yorke chaotic if and only if \(T\) is $(\delta_m)$-Li--Yorke chaotic.
		\item $T$ is mean Li--Yorke chaotic if and only if \(T\) is $\left(\frac{1}{m}\sum_{k=1}^{m}\delta_k\right)$-Li--Yorke chaotic.
	\end{enumerate}
\end{proposition}

\begin{proof}
	We prove the first equivalence. If $T$ is Li--Yorke chaotic, then $T$
	admits an irregular vector by \cite[Theorem~5]{BeBoMaPe}. Hence, by
	Proposition~3.2, $T$ is $(\delta_m)$-Li--Yorke chaotic.
	The converse implication is immediate.
	
	For the second equivalence, assume that $T$ is mean Li--Yorke chaotic.
	Then $T$ admits a mean irregular vector by \cite[Theorem~5]{BeBoPe20}.
	Proposition~3.2 again yields that $T$ is
	$\left(\frac{1}{m}\sum_{k=1}^{m}\delta_k\right)$-Li--Yorke chaotic.
	The converse implication is immediate.
\end{proof}

In the setting of Fr\'echet spaces, $T$
is $\left(\frac{1}{m}\sum_{k=1}^{m}\delta_k\right)$-Li--Yorke chaotic
if and only if $T$ is mean Li--Yorke $\beta$-extremely chaotic for some \(\beta\in \mathbb{N}\), a notion introduced by Jiang and Li in \cite{JiLi25}.

\begin{theorem}\label{weak:li-yorke}
	Let \(X\) be a Fréchet or Banach space. If $T$ is (densely) $(\mu_i)$-Li–Yorke chaotic, then $T$ is (densely) Li–Yorke chaotic.
\end{theorem}

\begin{proof}
	Suppose that $T$ is $(\mu_i)$--Li--Yorke chaotic, and let $\Gamma \subset X$ be the
	corresponding uncountable set together with $\beta \in \mathbb{N}$ associated
	with $(\mu_i)$.
	
	Fix any $x \neq y$ in $\Gamma$. We claim that
	$\displaystyle{\liminf_{n\to\infty} \rho(T^{n}x,T^{n}y)=0}$.
	Suppose, by contradiction, that
	$\displaystyle{\liminf_{n\to\infty} \rho(T^{n}x,T^{n}y) > 0}$.
	Then there exist $\delta>0$ and $n_{0}\in\mathbb{N}$ such that
	$\rho(T^{n}x,T^{n}y)>\delta$ for all $n>n_{0}$.
	For each $i\in I$, we have
	\begin{align*}
		\int \rho(T^{t}x,T^{t}y)\, d\mu_i(t)
		&\ge \min_{0\le k\le n_{0}} \rho(T^{k}x,T^{k}y)\,\mu_i([0,n_{0}])
		+ \delta \bigl(1-\mu_i([0,n_{0}])\bigr) \\
		&\ge \min\Bigl\{\delta,\min_{0\le k\le n_{0}} \rho(T^{k}x,T^{k}y)\Bigr\},
	\end{align*}
	which contradicts the $(\mu_i)$-irregularity of $x$.
	Hence,
	$\displaystyle{\liminf_{n\to\infty} \rho(T^{n}x,T^{n}y)=0}$.
	
	To conclude, we claim that
	$\displaystyle{\limsup_{n\to\infty} \|T^{n}(x-y)\|_{\beta} = \infty}$.
	Suppose otherwise, that is,
	$\displaystyle{\limsup_{n\to\infty} \|T^{n}(x-y)\|_{\beta} < \infty}$. Put \(z:=x-y\),
	then there exist $M>0$ and $m_{0}\in\mathbb{N}$ such that
	$\|T^{n}z\|_{\beta} \le M$ for all $n>m_{0}$.
	Hence, for each $i\in I$,
	\begin{align*}
		\int \|T^{t}z\|_{\beta}\, d\mu_i(t)
		&\le \max_{0\le k\le m_{0}} \|T^{k}z\|_{\beta}\,\mu_i([0,m_{0}])
		+ M\bigl(1-\mu_i([0,m_{0}])\bigr) \\
		&\le \max\Bigl\{M,\max_{0\le k\le m_{0}} \|T^{k}z\|_{\beta}\Bigr\},
	\end{align*}
	which again contradicts the $(\mu_i)$-irregularity of $x$.
	
	Therefore, $T$ is Li--Yorke chaotic. If $\Gamma$ is dense in $X$, it follows that
	$T$ is densely Li--Yorke chaotic.
\end{proof}

\begin{theorem}[Mycielski Theorem \cite{Mycielski64,TaXiLu}] \label{Myci}
	Suppose that \( X \) is a separable complete metric space without isolated points, and that for every \( n \in \mathbb{N} \), the set \( \mathcal{R}_{n} \) is residual in the product space \( X^{n} \). Then there is a Mycielski set \( \mathcal{K} \) in \( X \) such that
	\[
	(x_{1}, x_{2}, \ldots, x_{n}) \in \mathcal{R}_{n}
	\]
	for each \( n \in \mathbb{N} \) and any pairwise distinct \( n \) points \( x_{1}, x_{2}, \ldots, x_{n} \) in \( \mathcal{K} \).
\end{theorem}

A set \( \mathcal{K} \) is referred to as a Mycielski set if the intersection of \( \mathcal{K} \) and any nonempty open set \( U \) contains a Cantor set.

\begin{theorem}\label{equiv-dense}
	Let $(T_t)_{t\in\Lambda}\subset\mathcal{L}(X,Y)$, where $X$ is
	separable. Then the following assertions are equivalent:
	\begin{enumerate}
		\item[(1)] The family $(T_t)_{t\in\Lambda}$ is densely
		$(\mu_i)_{i\in I}$-Li--Yorke chaotic.
		
		\item[(2)] For some $\beta\in\mathbb{N}$, the family
		$(T_t)_{t\in\Lambda}$ admits a dense set of $\beta$-absolutely
		$(\mu_i)_{i\in I}$-irregular vectors.
		
		\item[(3)] For some $\beta\in\mathbb{N}$, the family
		$(T_t)_{t\in\Lambda}$ admits a residual set of $\beta$-absolutely
		$(\mu_i)_{i\in I}$-irregular vectors.
	\end{enumerate}
\end{theorem}

\begin{proof}
	Assume (1). Then there exist a dense uncountable set
	$\Gamma\subset X$ and $\beta\in\mathbb{N}$ such that $p-q$ is
	$\beta$-absolutely $(\mu_i)_{i\in I}$-irregular whenever
	$p,q\in\Gamma$ are distinct. Fix $q_0\in\Gamma$ and set
	\[
	E:=\{p-q_0:p\in\Gamma\setminus\{q_0\}\}.
	\]
	Then $E$ is dense in $X$ and consists of $\beta$-absolutely
	$(\mu_i)_{i\in I}$-irregular vectors. Thus (2) holds.
	
	Assume now (2). The set of $\beta$-absolutely
	$(\mu_i)_{i\in I}$-irregular vectors is a $G_\delta$-subset of $X$.
	Since it is dense, it is residual, and hence (3) follows.
	
	Finally, assume (3), and let $\mathcal{R}_\beta$ denote the
	set of $\beta$-absolutely $(\mu_i)_{i\in I}$-irregular vectors. Choose
	dense open sets $(U_\ell)_{\ell\in\mathbb{N}}$ in $X$ such that
	\[
	\bigcap_{\ell\in\mathbb{N}}U_\ell\subset\mathcal{R}_\beta.
	\]
	For each $\ell\in\mathbb{N}$, set
	\[
	W_\ell
	:=
	\{(p,q)\in X\times X:p-q\in U_\ell\}.
	\]
	The map
	\[
	\pi\colon X\times X\longrightarrow X,
	\qquad
	\pi(p,q)=p-q,
	\]
	is continuous, open, and surjective. Consequently,
	$W_\ell=\pi^{-1}(U_\ell)$ is a dense open subset of $X\times X$.
	Therefore,
	\[
	W:=\bigcap_{\ell\in\mathbb{N}}W_\ell
	\]
	is residual in $X\times X$.
	
	Under (3), the space $X$ is nontrivial and hence, being a
	separable $F$-space, is a perfect Polish space. By the Mycielski
	Theorem, there exists a dense Mycielski set $\mathcal{K}\subset X$ such that
	\[
	(\mathcal{K}\times\mathcal{K})\setminus\Delta\subset W,
	\]
	where $\Delta$ denotes the diagonal of $X\times X$. Thus, for every
	pair of distinct points $p,q\in\mathcal{K}$,
	\[
	p-q\in\bigcap_{\ell\in\mathbb{N}}U_\ell
	\subset\mathcal{R}_\beta.
	\]
	Hence $(T_t)_{t\in\Lambda}$ is densely
	$(\mu_i)_{i\in I}$-Li--Yorke chaotic.
\end{proof}

The following criterion is inspired by \cite[Definition~7]{BeBoMaPe}.

\begin{definition}\label{LYCC}
	We say that $(T_t)_{t\in\Lambda}\subset\mathcal{L}(X,Y)$ satisfies the
	\emph{$(\mu_i)_{i\in I}$-Li--Yorke chaos criterion} if there exist a set
	$X_0\subset X$ and $\beta\in\mathbb{N}$ such that the following
	conditions hold:
	\begin{enumerate}
		\item[(1)] There exists a sequence $(i_m)_{m\in\mathbb{N}}$ in $I$
		with $i_m\to\infty$ such that
		\[
		\int_{\Lambda}\rho(T_t x,0)\,d\mu_{i_m}(t)
		\longrightarrow 0
		\qquad\text{for every }x\in X_0.
		\]
		
		\item[(2)] There exist a bounded sequence
		$(y_k)_{k\in\mathbb{N}}$ in
		$\overline{\operatorname{span}}(X_0)$ and a sequence
		$(j_k)_{k\in\mathbb{N}}$ in $I$, with $j_k\to\infty$, such that
		\[
		\int_{\Lambda}\|T_t y_k\|_{\beta}\,d\mu_{j_k}(t)>k
		\qquad\text{for every }k\in\mathbb{N}.
		\]
	\end{enumerate}
	
	If, in addition, $X_0$ is dense in $X$, we say that
	$(T_t)_{t\in\Lambda}$ satisfies the
	\emph{dense $(\mu_i)_{i\in I}$-Li--Yorke chaos criterion}.
\end{definition}

\begin{theorem}\label{equiv-LY-criterion} 
	A family \((T_{t})_{t\in \Lambda}\subset \mathcal{L}(X,Y)\) is \((\mu_{i})\)\textit{-Li--Yorke chaotic} if and only if it satisfies the \((\mu_{i})\)\textit{-Li--Yorke chaos criterion}. 
\end{theorem}

\begin{proof}
	Assume first that $(T_t)_{t\in\Lambda}$ is
	$(\mu_i)_{i\in I}$-Li--Yorke chaotic. Then it admits a
	$\beta$-absolutely $(\mu_i)_{i\in I}$-irregular vector $p$ for some
	$\beta\in\mathbb{N}$. Hence there exist sequences $(i_m)_{m\in\mathbb N}$
	and $(j_k)_{k\in\mathbb N}$ in $I$, both tending to infinity, such that
	\[
	\int_\Lambda \rho(T_t p,0)\,d\mu_{i_m}(t)\longrightarrow 0
	\]
	and
	\[
	\int_\Lambda\|T_t p\|_\beta\,d\mu_{j_k}(t)>k
	\qquad (k\in\mathbb N).
	\]
	Thus the criterion holds with $X_0:=\{p\}$ and $y_k:=p$ for every
	$k\in\mathbb N$.
	
	Conversely, suppose that $(T_t)_{t\in\Lambda}$ satisfies the
	$(\mu_i)_{i\in I}$-Li--Yorke chaos criterion, and set
	\[
	Z:=\overline{\operatorname{span}}(X_0).
	\]
	Then $Z$, endowed with the metric induced by $\mathrm D$, is an
	$F$-space. For $t\in\Lambda$, let
	\[
	S_t:=\left.T_t\right|_Z.
	\]
	We first show that $(S_t)_{t\in\Lambda}$ is not
	$\beta$-absolutely $(\mu_i)_{i\in I}$-bounded. Let $\delta>0$ and
	$M>0$. Since $(y_k)_{k\in\mathbb N}$ is bounded, there exists
	$\ell\in\mathbb N$ such that
	\[
	\mathrm D\left(\frac{1}{\ell}y_k,0\right)<\delta
	\qquad\text{for every }k\in\mathbb N.
	\]
	Choosing $k>\ell M$, we obtain
	\[
	\int_\Lambda
	\left\|S_t\left(\frac{1}{\ell}y_k\right)\right\|_\beta
	\,d\mu_{j_k}(t)
	=
	\frac{1}{\ell}
	\int_\Lambda\|S_t y_k\|_\beta\,d\mu_{j_k}(t)
	>
	\frac{k}{\ell}>M.
	\]
	Therefore,
	\[
	\sup_{\substack{z\in Z\\ \mathrm D(z,0)<\delta}}
	\sup_{i\in I}
	\int_\Lambda\|S_t z\|_\beta\,d\mu_i(t)
	=\infty.
	\]
	By Corollary~\ref{not-abs-bounded}, the set
	\[
	\mathcal B:=
	\left\{
	x\in Z:
	\sup_{i\in I}
	\int_\Lambda\|S_t x\|_\beta\,d\mu_i(t)=\infty
	\right\}
	\]
	is residual in $Z$.
	
	For $a\in I$, let
	\[
	K_a:=
	\overline{
		\bigcup_{\substack{i\in I\\ i\leq a}}
		\operatorname{supp}(\mu_i)
	}.
	\]
	By \eqref{compact-soporte}, $K_a$ is compact, and hence, for every
	$x\in Z$,
	\[
	\sup_{\substack{i\in I\\ i\leq a}}
	\int_\Lambda\|S_t x\|_\beta\,d\mu_i(t)
	\leq
	\sup_{t\in K_a}\|S_t x\|_\beta
	<\infty.
	\]
	Consequently,
	\[
	\mathcal B
	=
	\left\{
	x\in Z:
	\limsup_{\substack{i\to\infty\\ i\in I}}
	\int_\Lambda\|S_t x\|_\beta\,d\mu_i(t)=\infty
	\right\}.
	\]
	Now set
	\[
	\mathcal A:=
	\left\{
	x\in Z:
	\liminf_{m\to\infty}
	\int_\Lambda\rho(S_t x,0)\,d\mu_{i_m}(t)=0
	\right\}.
	\]
	Condition~(i) in Definition~\ref{LYCC}, together with linearity, implies that \[ \text{span}(X_0)\subset\mathcal A. \] Since both $\mathcal A$ and $\mathcal B$ are residual in $Z$, we may choose $x\in\mathcal A\cap\mathcal B$. Then $x$ is a $\beta$-absolutely $(\mu_i)_{i\in I}$-irregular vector for $(S_t)_{t\in\Lambda}$, and hence also for $(T_t)_{t\in\Lambda}$. The conclusion follows from Proposition~\ref{LY-irr}.
\end{proof}

\begin{theorem}\label{prod-LY}
	Let $(T_t)_{t\in\Lambda}\subset\mathcal{L}(X,Y)$, where $X$ is
	separable. Then the following assertions are equivalent:
	\begin{enumerate}
		\item[(1)] The family $(T_t)_{t\in\Lambda}$ satisfies the dense
		$(\mu_i)_{i\in I}$-Li--Yorke chaos criterion.
		
		\item[(2)] For every $n\in\mathbb{N}$, the family
		\[
		\left(\bigoplus_{k=1}^{n}T_t\right)_{t\in\Lambda}
		\]
		is densely $(\mu_i)_{i\in I}$-Li--Yorke chaotic.
	\end{enumerate}
\end{theorem}

\begin{proof}
	Assume first that $(T_t)_{t\in\Lambda}$ satisfies the dense
	$(\mu_i)_{i\in I}$-Li--Yorke chaos criterion. Let
	$X_0\subset X$, $\beta\in\mathbb{N}$, $(i_m)_{m\in\mathbb{N}}$,
	$(j_k)_{k\in\mathbb{N}}$, and $(y_k)_{k\in\mathbb{N}}$ be as in
	Definition~\ref{LYCC}. For $n\in\mathbb{N}$, set
	\[
	T_t^{(n)}:=\bigoplus_{k=1}^{n}T_t
	\colon X^n\longrightarrow Y^n.
	\]
	With respect to the corresponding product metric and seminorm, the
	family $(T_t^{(n)})_{t\in\Lambda}$ satisfies the dense
	$(\mu_i)_{i\in I}$-Li--Yorke chaos criterion with the dense set
	$X_0^n$. Indeed, the first condition follows coordinatewise, while
	the second one follows by considering the bounded sequence
	\[
	\bigl(y_k,0,\ldots,0\bigr)_{k\in\mathbb{N}}
	\subset
	\overline{\operatorname{span}}(X_0^n).
	\]
	Arguing as in the proof of Theorem~\ref{equiv-LY-criterion}, the
	family $(T_t^{(n)})_{t\in\Lambda}$ admits a residual set of
	$\beta$-absolutely $(\mu_i)_{i\in I}$-irregular vectors in $X^n$.
	Theorem~\ref{equiv-dense} now shows that it is densely
	$(\mu_i)_{i\in I}$-Li--Yorke chaotic. Thus (2) holds.
	
	Conversely, assume (2). For each $n\in\mathbb{N}$, let
	\[
	\mathcal{R}_n
	:=
	\left\{
	(x_1,\ldots,x_n)\in X^n:
	\begin{array}{c}
		\text{there exists }(j_k)_{k\in\mathbb{N}}\subset I,
		\ j_k\to\infty,\\[2pt]
		\displaystyle
		\int_\Lambda\rho(T_t x_\ell,0)\,d\mu_{j_k}(t)
		\longrightarrow 0
		\quad\text{for }1\leq\ell\leq n
	\end{array}
	\right\}.
	\]
	By Proposition~\ref{dist-near to 0}, each $\mathcal{R}_n$ is residual
	in $X^n$. The Mycielski theorem therefore yields a dense Mycielski set
	$\mathcal K\subset X$ such that every $n$-tuple of pairwise distinct
	points of $\mathcal K$ belongs to $\mathcal R_n$.
	
	Choose a sequence of distinct points
	$(q_\ell)_{\ell\in\mathbb{N}}\subset\mathcal K$ which is dense in
	$X$. We recursively select a sequence $(i_k)_{k\in\mathbb{N}}$ in
	$I$, with $i_k\to\infty$, such that
	\[
	\max_{1\leq\ell\leq k}
	\int_\Lambda\rho(T_t q_\ell,0)\,d\mu_{i_k}(t)
	<\frac{1}{k}.
	\]
	This is possible since
	$(q_1,\ldots,q_k)\in\mathcal R_k$ for every $k$. Consequently,
	\[
	\int_\Lambda\rho(T_t q_\ell,0)\,d\mu_{i_k}(t)
	\longrightarrow 0
	\qquad\text{for every }\ell\in\mathbb{N}.
	\]
	Thus condition~(1) in Definition~\ref{LYCC} holds for
	\[
	X_0:=\{q_\ell:\ell\in\mathbb{N}\},
	\]
	which is dense in $X$.
	
	It remains to verify condition~(2) of the criterion. By the
	case $n=1$ of the assumption, $(T_t)_{t\in\Lambda}$ is densely
	$(\mu_i)_{i\in I}$-Li--Yorke chaotic. Hence it admits a
	$\beta$-absolutely $(\mu_i)_{i\in I}$-irregular vector $p$ for some
	$\beta\in\mathbb{N}$. Since
	\[
	\overline{\operatorname{span}}(X_0)=X,
	\]
	the constant sequence $y_k:=p$ is a bounded sequence in
	$\overline{\operatorname{span}}(X_0)$. Moreover, there exists a
	sequence $(j_k)_{k\in\mathbb{N}}$ in $I$, with $j_k\to\infty$, such
	that
	\[
	\int_\Lambda\|T_t y_k\|_\beta\,d\mu_{j_k}(t)>k
	\qquad\text{for every }k\in\mathbb{N}.
	\]
	Therefore, $(T_t)_{t\in\Lambda}$ satisfies the dense
	$(\mu_i)_{i\in I}$-Li--Yorke chaos criterion.
\end{proof}



\section{\(\mathrm{D}\)-phenomenon \(\Psi^{\mathrm{LY}}\) and Furstenberg--Borel families} \label{D-phe}

In this section we assume that the map
\begin{align}\label{continuous}
	(t,x)\in \Lambda\times X \longmapsto T_t x \in Y
\end{align}
is continuous.

\begin{definition}\label{Furs-Borel}
	Let \(\Lambda\) be a metric space and let \(\mathcal{B}(\Lambda)\) denote the Borel \(\sigma\)-algebra of \(\Lambda\). A collection \(\mathcal{F} \subset \mathcal{B}(\Lambda)\) is called a \emph{Furstenberg--Borel family on \(\Lambda\)} if it is upward closed within \(\mathcal{B}(\Lambda)\); that is,
	\begin{equation}\label{eq:furstenberg_property}
		A \in \mathcal{F}, \; B \in \mathcal{B}(\Lambda) \text{ and } A \subset B 
		\quad \Longrightarrow \quad 
		B \in \mathcal{F}.
	\end{equation}
	We say that \(\mathcal{F}\) is \emph{proper} if \(\emptyset \notin \mathcal{F}\).
\end{definition}

Observe that when \(\Lambda=\mathbb{N}\) is endowed with the discrete topology, this notion reduces to the classical concept of a Furstenberg family.

\begin{definition}[\cite{ArSa}]\label{def:D-phenomenon}
	Let $X$ and $Y$ be topological vector spaces, and let $\Psi$ be a set-valued 
	mapping on $X$ such that, for every $x \in X$, $\Psi(x)$ is a non-empty 
	family of non-empty open subsets of $Y$. Then $\Psi$ is called a 
	\(\mathrm{D}\)-\emph{phenomenon} from $X$ to $Y$ if, for every finite linear 
	combination $p = \sum_{j=1}^n \alpha_j x_j$ with $\alpha_j \neq 0$, and every 
	$V \in \Psi(p)$, there exist sets $W_j \in \Psi(x_j)$ such that
	\[
	\sum_{j=1}^n \alpha_j W_j \subset V.
	\]
\end{definition}

Let \((T_t)_{t\in\Lambda}\subset\mathcal{L}(X,Y)\) be a family of operators, let \(\mathcal{F}\) be a Furstenberg--Borel family on \(\Lambda\), and let \(\Psi\) be a \(D\)-phenomenon from \(X\) to \(Y\). We define
\begin{equation}\label{eq:F-Psi-set}
	\mathcal{F}\Psi((T_t)_{t\in\Lambda})
	:=
	\left\{
	x\in X :
	\forall\, V\in\Psi(x),\;
	\{t\in\Lambda : T_t x \in V\}\in\mathcal{F}
	\right\}.
\end{equation}

By assumption \eqref{continuous}, for every \(x\in X\) and every non-empty
open set \(V\subset Y\), the set of return times
\[
\{t\in\Lambda : T_t x \in V\}
\]
is open in \(\Lambda\), and therefore belongs to \(\mathcal{B}(\Lambda)\).

We define the \emph{family of upper density one} associated with \((\mu_i)\) by
\begin{equation}\label{eq:measure_family}
	\mathcal{F}_{(\mu_i)} := \left\{ A \in \mathcal{B}(\Lambda) : \limsup_{i \to \infty} \mu_i(A) = 1 \right\}.
\end{equation}
It is straightforward to verify that \(\mathcal{F}_{(\mu_i)}\) is a proper Furstenberg--Borel family.

For any \(\beta\in\mathbb{N}\), we define the \(D\)-phenomenon \(\Psi_{\beta}^{\mathrm{LY}}\) from \(X\) to \(Y\) by
\[
\Psi_{\beta}^{\mathrm{LY}}(x)
:=
\bigl\{
\{y\in Y : \rho(y,0)<r\},\;
\{y\in Y : \|y\|_{\beta}>s\}
: r,s>0
\bigr\}.
\]

In the metrizable locally convex
space \(Y\) we use the seminorms
\((\|\cdot\|_\beta)_{\beta\in\mathbb N}\), which leads to the
family of \(\mathrm{D}\)-phenomena \(\Psi^{\mathrm{LY}}_\beta\).
In the normed space setting, where a single norm is available,
we simply write \(\Psi^{\mathrm{LY}}\).

In \cite{ArSa}, A. Arbieto and M. Saavedra showed that in the discrete setting, that is, when \(T\in\mathcal{L}(X)\) and \(X\) is a Banach space, one has
\begin{align*}
	\mathcal{F}_{(\mu_m)}\Psi^{\mathrm{LY}}(T)
	=
	\begin{cases}
		\{\text{all irregular vectors for } T\}
		&\text{if } \mu_m=\delta_m, \\[6pt]
		\{\text{all distrib. irregular vectors for } T\}
		&\text{if } \displaystyle \mu_m=\frac{1}{m}\sum_{k=1}^{m}\delta_k .
	\end{cases}
\end{align*}

\begin{definition}
	Let \((T_t)_{t\in\Lambda}\subset\mathcal{L}(X,Y)\), let \(\beta\in\mathbb{N}\), and let \((\mu_i)\) be as above.  
	A vector \(x\in X\) belonging to
	\[
	\mathcal{F}_{(\mu_i)}\Psi_{\beta}^{\mathrm{LY}}\big((T_t)_{t\in\Lambda}\big)
	\]
	is called an \emph{\(\beta\)-distributionally \((\mu_i)\)-irregular} vector for the family \((T_t)_{t\in\Lambda}\).
	
	We say that \((T_t)_{t\in\Lambda}\) is \emph{\((\mu_i)\)-distributionally chaotic} if there exist an uncountable set \(\Gamma\subset X\) and \(\beta\in\mathbb{N}\) such that, for every distinct \(p,q\in \Gamma\),
	\[
	p-q\in \mathcal{F}_{(\mu_i)}\Psi_{\beta}^{\mathrm{LY}}\big((T_t)_{t\in\Lambda}\big).
	\]
	
	If, in addition, \(\Gamma\) can be chosen dense in \(X\), we say that \((T_t)_{t\in\Lambda}\) is \emph{densely \((\mu_i)\)-distributionally chaotic}.
\end{definition}

\begin{proposition}\label{dist-G-delta}
	\(\mathcal{F}_{(\mu_{i})}\Psi^{\mathrm{LY}}_{\beta}((T_{t})_{t\in \Lambda})\)
	is a \(G_\delta\)-subset of \(X\).
\end{proposition}

\begin{proof}
	For each \(k\in \mathbb{N}\), define
	\begin{align*}
		S_{k} &:= \left\{ x\in X : \exists i\in I,\ i\ge k \text{ s.t }
		\mu_{i}\big(\{t\in \Lambda : \rho(T_{t}x,0)<\frac{1}{k}\}\big)>1-\frac{1}{k} \right\},\\
		R_{k, \beta} &:= \left\{ x\in X : \exists i\in I,\ i\ge k \text{ s.t }
		\mu_{i}\big(\{t\in \Lambda : \|T_{t}x\|_{\beta}>k\}\big)>1-\frac{1}{k} \right\}.
	\end{align*}
	
	We claim that \(S_{k}\) and \(R_{k,\beta}\) are open. We prove it for \(S_{k}\); the argument for \(R_{k,\beta}\) is analogous.
	
	Let \(x\in S_{k}\). Then there exist \(i\ge k\) and, by the regularity of the measure \(\mu_i\), a compact set \(A\subset \Lambda\) such that
	\[
	A \subset \{t\in \Lambda : \rho(T_{t}x,0)<k^{-1}\}
	\quad \text{and} \quad
	\mu_{i}(A)>1-k^{-1}.
	\]
	Since \((T_{t})_{t\in A}\) is equicontinuous, there exists a neighborhood \(U\) of \(0\) in \(X\) such that
	\[
	\rho(T_{t}y,0) < \frac{1}{k} - \max_{t\in A}\rho(T_{t}x,0),
	\quad \forall\, y\in U,\ \forall\, t\in A.
	\]
	It follows that
	\[
	\rho(T_{t}(x+y),0) < \frac{1}{k}, \quad \forall\, y\in U,\ \forall\, t\in A,
	\]
	and hence
	\[
	A \subset \{t\in \Lambda : \rho(T_{t}(x+y),0)<\frac{1}{k}\}, \quad \forall\, y\in U.
	\]
	Therefore,
	\[
	\mu_{i}\big(\{t\in \Lambda : \rho(T_{t}(x+y),0)<k^{-1}\}\big)
	\ge \mu_{i}(A) > 1-k^{-1},
	\]
	which shows that \(x+U \subset S_{k}\). Thus, \(S_{k}\) is open.
	
	Finally,
	\[
	\mathcal{F}_{(\mu_{i})}\Psi^{\mathrm{LY}}_{\beta}((T_{t})_{t\in \Lambda})
	=
	\bigcap_{k\in \mathbb{N}} S_{k} \cap \bigcap_{k\in \mathbb{N}} R_{k,\beta},
	\]
	and the conclusion follows.
\end{proof}

\begin{remark}
	The proof identifies two natural \(G_\delta\)-sets, namely
	\(\bigcap_{k\in \mathbb{N}} S_{k}\) and \(\bigcap_{k\in \mathbb{N}} R_{k,\beta}\).
	In the discrete setting, when \(T\in \mathcal{L}(X)\) with \(X\) a Banach space and
	\(\mu_{m}:=\frac{1}{m}\sum_{\ell=1}^{m}\delta_{\ell}\), these sets correspond,
	respectively, to the set of vectors which are distributionally close to zero
	and to the set of vectors with distributionally \(\beta\)-unbounded orbit.
\end{remark}	

\begin{remark}\label{form distri-irre}
	From the previous proof, it follows that \(x\in \bigcap_{k\in \mathbb{N}} S_{k}\) if and only if there exist a sequence \((i_{m})_{m}\) in \(I\) with \(i_{m}\to\infty\), and a sequence of compact sets \((A_{m})_{m}\) in \(\Lambda\) such that \(\mu_{i_{m}}(A_{m})>1-\frac{1}{m}\) for each \(m\in \mathbb{N}\) and
	\begin{equation}\label{near 0}
		\lim_{m\to\infty} \max_{t\in A_{m}} \rho(T_{t}x,0)=0.
	\end{equation}

	Similarly, \(x\in \bigcap_{k\in \mathbb{N}} R_{k,\beta}\) if and only if there exist a sequence \((j_{m})_{m}\) in \(I\) with \(j_{m}\to\infty\), and a sequence of compact sets \((B_{m})_{m}\) in \(\Lambda\) such that \(\mu_{j_{m}}(B_{m})>1-\frac{1}{m}\) for each \(m\in \mathbb{N}\) and
	\begin{equation}\label{beta-unbo}
		\lim_{m\to\infty} \min_{t\in B_{m}} \|T_{t}x\|_{\beta}=\infty.
	\end{equation}
\end{remark}	

\begin{proposition}\label{equi dis-cha dist-irr}
	$(T_{t})_{t\in \Lambda}$ is $(\mu_{i})$-distributionally chaotic if and only if it admits a \(\beta\)-distributionally \((\mu_{i})\)-irregular vector for some \(\beta\in \mathbb{N}\).
\end{proposition}

\begin{proof}
	Assume that $(T_{t})_{t\in \Lambda}$ admits a \(\beta\)-distributionally  $(\mu_{i})$-irregular vector $z\in X$, and set $\Gamma:=\mathbb{K}z$. 
	For any distinct $p,q\in \Gamma$, there exists $a\in\mathbb{K}\setminus\{0\}$ such that $p-q=a z$.
	
	By Remark~\ref{form distri-irre}, there exist sequences $(i_m)_m$, $(j_m)_m\subset I$ with $i_m,j_m\to\infty$, and sequences of compact sets $(A_m)_m$, $(B_m)_m\subset \Lambda$  such that \(\mu_{i_{m}}(A_{m})>1-\frac{1}{m}\), \(\mu_{j_{m}}(B_{m})>1-\frac{1}{m}\) for every \(m\in \mathbb{N}\) and
	\[
	\lim_{m\to\infty}\max_{t\in A_m}\rho(T_t z,0)=0,
	\qquad
	\lim_{m\to\infty}\min_{t\in B_m}\|T_t z\|_\beta=\infty.
	\]
	It follows that
	\[
	\max_{t\in A_m}\rho(T_t(a z),0)
	\le (|a|+1)\max_{t\in A_m}\rho(T_t z,0)\xrightarrow[m\to\infty]{}0,
	\]
	and
	\[
	\min_{t\in B_m}\|T_t(a z)\|_\beta
	=|a|\min_{t\in B_m}\|T_t z\|_\beta \xrightarrow[m\to\infty]{}\infty.
	\]
	Since \(\Gamma\) is uncountable, \((T_t)_{t\in\Lambda}\) is \((\mu_i)\)-distributionally chaotic. 
	The converse implication is immediate.
\end{proof}

\begin{theorem}\label{dense-mu_i-distri}
	Let $(T_t)_{t\in \Lambda}\subset \mathcal{L}(X,Y)$ with $X$ separable.  
	The following assertions are equivalent:
	\begin{enumerate}
		\item $(T_t)_t$ is densely $(\mu_i)$-distributionally chaotic;
		
		\item there exists $\beta\in\mathbb{N}$ such that 
		\[
		\mathcal{F}_{(\mu_i)}\Psi_{\beta}^{\mathrm{LY}}((T_t)_{t\in \Lambda})
		\]
		is dense in $X$.
		\item there exists $\beta\in\mathbb{N}$ such that 
		\[
		\mathcal{F}_{(\mu_i)}\Psi_{\beta}^{\mathrm{LY}}((T_t)_{t\in \Lambda})
		\]
		is residual in $X$.
	\end{enumerate}
\end{theorem}

\begin{proof}
	Assume that $(T_t)_t$ is densely $(\mu_i)$-distributionally chaotic. Then there exist an uncountable dense set $\Gamma\subset X$ and $\beta\in\mathbb{N}$ such that, for every distinct $p,q\in\Gamma$,
	\[
	p-q \in \mathcal{F}_{(\mu_i)}\Psi_{\beta}^{\mathrm{LY}}((T_t)_{t\in \Lambda}).
	\]
	Fix $q_0\in\Gamma$ and set
	\[
	E:=\{p-q_0 : p\in \Gamma\setminus\{q_0\}\}.
	\]
	Then $E$ is dense in $X$ and $E\subset \mathcal{F}_{(\mu_i)}\Psi_{\beta}^{\mathrm{LY}}((T_t)_{t\in \Lambda})$, which proves (2).
	
	The implication (2) $\Rightarrow$ (3) follows from Proposition~\ref{dist-G-delta}, since $\mathcal{F}_{(\mu_i)}\Psi_{\beta}^{\mathrm{LY}}((T_t))$ is a $G_\delta$ subset of $X$.
	
	Assume now that (3) holds. By Proposition~\ref{dist-G-delta}, there exists a sequence of dense open sets $(U_\ell)_{\ell\in\mathbb{N}}$ such that
	\[
	\bigcap_{\ell\in\mathbb{N}} U_\ell=
	\mathcal{F}_{(\mu_i)}\Psi_{\beta}^{\mathrm{LY}}((T_t)_{t\in \Lambda}).
	\]
	For each $\ell\in\mathbb{N}$, define \(	W_\ell:=\{(p,q)\in X\times X : p-q \in U_\ell\}.\)
	Then each $W_\ell$ is a dense open subset of $X\times X$, and hence \(	\bigcap_{\ell} W_\ell\)
	is residual in $X\times X$. By Mycielski's theorem, there exists a Mycielski set $\Gamma\subset X$ such that
	\[
	\Gamma\times \Gamma \setminus \Delta \subset \bigcap_{\ell\in\mathbb{N}} W_\ell.
	\]
	Therefore, for every distinct $p,q\in \Gamma$,
	\[
	p-q \in \mathcal{F}_{(\mu_i)}\Psi_{\beta}^{\mathrm{LY}}((T_t)_{t\in \Lambda}),
	\]
	which shows that $(T_t)_t$ is densely $(\mu_i)$-distributionally chaotic.
\end{proof}

Let \(\Psi^{\mathbf{0}}\) be the \(D\)-phenomenon from \(X\) to \(Y\) defined by
\[
\Psi^{\mathbf{0}}(x):=\{U\subset Y : U \text{ is open and } 0\in U\}, \quad x\in X.
\]

\begin{definition}
	A vector \(x\in X\) is said to be \((\mu_i)\)-distributionally near to \(0\) for \((T_t)_{t\in \Lambda}\) if \(x \in \bigcap_{k} S_k\).
	Similarly, \(x\in X\) is said to be \(\beta\)-distributionally \((\mu_i)\)-unbounded for \((T_t)_{t\in \Lambda}\) if \(x \in \bigcap_{k} R_{k,\beta}.\)
\end{definition}

The following identification is immediate.

\begin{proposition}\label{dist-near to 0}
	We have
	\[
	\mathcal{F}_{(\mu_{i})}\Psi^{\mathbf{0}}((T_{t})_{t\in \Lambda})
	=
	\{x\in X : x \text{ is } (\mu_i)\text{-distributionally near to } 0 \text{ for } (T_t)_{t\in \Lambda}\}.
	\]
	Moreover, such vectors satisfy \eqref{near 0}.
\end{proposition}

\begin{theorem}\label{prod-dis-near 0}
	Let \(X\) be a separable space. The following statements are equivalent:
	\begin{enumerate}
		\item For each \(n\in \mathbb{N}\), \(\displaystyle{\mathcal{F}_{(\mu_{i})}\Psi^{\mathbf{0}}\left((\bigoplus_{k=1}^{n}T_{t})_{t\in \Lambda}\right)}\)
		is dense in \(X^{n}\).
		
		\item There exist a sequence \((i_{m})_{m}\) in \(I\) with \(i_{m}\to\infty\), a sequence of compact sets \((A_{m})_{m}\) in \(\Lambda\), and a dense subspace \(X_{0}\subset X\) such that \(\mu_{i_{m}}(A_{m})>1-\frac{1}{m}\) for each \(m\in \mathbb{N}\) and
		\[
		\lim_{m\to\infty} \max_{t\in A_{m}}\rho(T_{t}x,0)=0, \quad \forall x\in X_{0}.
		\]
	\end{enumerate}
\end{theorem}

\begin{proof}
	The implication \((2)\Rightarrow(1)\) is immediate. 
	We prove \((1)\Rightarrow(2)\).
	
	By assumption and Proposition~\ref{dist-near to 0}, for each 
	\(n\in\mathbb N\),
	\[
	\mathcal{F}_{(\mu_{i})}\Psi^{\mathbf{0}}\left((\bigoplus_{k=1}^{n}T_{t})_{t\in \Lambda}\right)
	\]
	is residual in \(X^{n}\).
	
	By Mycielski's theorem, there exists a Mycielski set 
	\(\mathcal K\subset X\) such that
	\[
	\mathcal K^{n}
	\subset 
	\mathcal{F}_{(\mu_{i})}\Psi^{\mathbf{0}}\left((\bigoplus_{k=1}^{n}T_{t})_{t\in \Lambda}\right)
	\quad \forall\; n\in\mathbb N.
	\]
	
	Let \(\{y_{\ell}\}_{\ell\in\mathbb N}\subset \mathcal K\) 
	be a countable dense subset. One can construct a sequence 
	\((A_{m})_{m\in\mathbb N}\) of compact subsets of \(\Lambda\) and a strictly increasing sequence 
	\((i_{m})_{m\in\mathbb N}\) in \(I\) with \(i_{m}\to \infty\) such that
	\[
	\mu_{i_{m}}(A_{m})>1-\frac{1}{m}
	\]
	and
	\[
	\max_{t\in A_{m}}
	\left(
	\max_{1\le \ell\le m}
	\rho(T_{t}y_{\ell},0)
	\right)
	<\frac{1}{m},
	\qquad m\in\mathbb N.
	\]
	In particular, for every \(\ell\in\mathbb N\),
	\[
	\lim_{m\to\infty}
	\max_{t\in A_{m}}\rho(T_{t}y_{\ell},0)=0.
	\]
	
	Set \(X_{0}:=\mathrm{span}\{y_{\ell}:\ell\in\mathbb N\}\). Then \(X_{0}\) is dense in \(X\), and the above estimate extends to every \(x\in X_{0}\). This completes the proof.
\end{proof}

\begin{proposition}\label{dist-unbounded}
	The following assertions are equivalent:
	\begin{enumerate}
		\item There exist \(\varepsilon>0\), a sequence 
		\((x_m)_{m\in\mathbb N}\subset X\), and a sequence 
		\((i_m)_{m\in\mathbb N}\) in \(I\) with \(i_m\to\infty\) such that 
		\(x_m\to 0\) and
		\[
		\int_{\Lambda} \mathbf{1}_{(\varepsilon,\infty)}
		\bigl(\|T_{t}x_m\|_{\beta}\bigr)\, d\mu_{i_m}(t)
		\;\ge\; 1-\frac{1}{m},
		\qquad m\in\mathbb N.
		\]
		
		\item There exists a \(\beta\)-distributionally \((\mu_i)\)-unbounded vector for \((T_t)_{t\in \Lambda}\).
		
		\item The set of all \(\beta\)-distributionally \((\mu_i)\)-unbounded vectors for \((T_t)_{t\in \Lambda}\) is residual in \(X\).
	\end{enumerate}
\end{proposition}

\begin{proof}
	The implication (3) $\Rightarrow$ (2) is immediate.
	
	Assume that (2) holds. Then there exists $x\in X$ which is 
	$\beta$-distributionally $(\mu_i)$-unbounded for $(T_t)_{t\in\Lambda}$. 
	By \eqref{beta-unbo}, there exist a sequence $(j_m)_m$ in $I$ with 
	$j_m\to\infty$ and a sequence of compact sets \((B_m)_m\) in \(\Lambda\) such that
	\[
	\mu_{j_m}(B_m)>1-\frac{1}{m}
	\quad \forall m\in\mathbb N,
	\qquad\text{and}\qquad
	\lim_{m\to\infty}\min_{t\in B_m}\|T_t x\|_\beta=\infty.
	\]
	
	Fix $\varepsilon>0$ and define $x_k:=x/k$ for $k\in\mathbb N$. Then $x_k\to 0$. 
	Given $k$, choose $m_0=m_0(k)$ such that
	\[
	\min_{t\in B_m}\|T_t x\|_\beta > k\varepsilon
	\qquad \text{for all } m\ge m_0.
	\]
	It follows that, for $m\ge m_0$,
	\begin{align*}
		\int_{\Lambda} \mathbf{1}_{(\varepsilon,\infty)}\bigl(\|T_t x_k\|_\beta\bigr)\, d\mu_{j_m}(t)
		&=
		\int_{\Lambda} \mathbf{1}_{(k\varepsilon,\infty)}\bigl(\|T_t x\|_\beta\bigr)\, d\mu_{j_m}(t) \\
		&\ge \mu_{j_m}(B_m).
	\end{align*}
	Since $\mu_{j_m}(B_m)\to 1$, we can extract a subsequence $(i_k)_k$ such that (1) holds.
	
	\medskip
	
	Now, we prove that (1) implies (3). For each $k\in\mathbb N$, we rewrite \(R_{k,\beta}\) as
	\[
	R_{k,\beta}
	=
	\left\{
	x\in X: \exists i\in I, i\ge k\; \text{such that}
	\int_{\Lambda} \mathbf{1}_{(k,\infty)}\bigl(\|T_t x\|_\beta\bigr)\, d\mu_i(t)
	> 1-\frac{1}{k}
	\right\}.
	\]
	
	We claim that $R_{k,\beta}$ is dense in $X$. Let $x\in X$ and let $U$ be a balanced neighborhood of $0$ in $X$ (see \cite[Theorem 1.14]{Ru}). By (1), there exist $u\in \frac{1}{4k}U$ and $i\in I$, $i\ge k$, such that
	\[
	\int_{\Lambda} \mathbf{1}_{(k,\infty)}\bigl(\|T_t u\|_\beta\bigr)\, d\mu_i(t)
	> 1-\frac{1}{2k}.
	\]
	
	Define $x_s:=x+2su$ for $s=0,\dots,2k-1$, so that $x_s\in x+U$. Set
	\[
	D:=\{t\in \operatorname{supp}(\mu_i):\ \|T_t u\|_\beta>k\},
	\]
	so that $\mu_i(D)>1-\frac{1}{2k}$. For each $s$, let
	\[
	B_s:=\{t\in \operatorname{supp}(\mu_i):\ \|T_t x_s\|_\beta\le k\}.
	\]
	
	We claim that the sets $B_s\cap D$ are pairwise disjoint. Indeed, if 
	$t_0\in B_s\cap B_r\cap D$ with $s\neq r$, then
	\[
	\|T_{t_0}x_s - T_{t_0}x_r\|_\beta
	\le 2k,
	\]
	whereas
	\[
	\|T_{t_0}x_s - T_{t_0}x_r\|_\beta
	=2|s-r|\,\|T_{t_0}u\|_\beta
	>2k,
	\]
	a contradiction.
	
	Thus the sets $B_s\cap D$ are disjoint, and hence for some $s_0$,
	\[
	\mu_i(B_{s_0}\cap D)
	\le \frac{1}{2k}\mu_i(D).
	\]
	Consequently,
	\[
	\mu_i(D\setminus B_{s_0})
	\ge \mu_i(D)-\mu_i(B_{s_0}\cap D)
	> \left(1-\frac{1}{2k}\right)^2
	> 1-\frac{1}{k}.
	\]
	
	For $t\in D\setminus B_{s_0}$ we have $\|T_t x_{s_0}\|_\beta>k$, and therefore
	\[
	\int_{\Lambda} \mathbf{1}_{(k,\infty)}\bigl(\|T_t x_{s_0}\|_\beta\bigr)\, d\mu_i(t)
	\ge \mu_i(D\setminus B_{s_0})
	> 1-\frac{1}{k}.
	\]
	Hence $x_{s_0}\in R_{k,\beta}$, proving that $R_{k,\beta}$ is dense.
	
	It follows that $\bigcap_{k} R_{k,\beta}$ is residual in $X$, which completes the proof.
\end{proof}

\begin{definition}\label{criterio-distrib}
	A family $(T_{t})_{t\in \Lambda}$ satisfies the \emph{$(\mu_i)$-distributional chaos criterion} if there exist a subset \(X_{0}\subset X\) and \(\beta\in \mathbb{N}\) such that:
	\begin{enumerate}
		\item There exist a sequence of compact sets $(A_{m})_{m}$ in \(\Lambda\) and a sequence $(i_{m})_{m}\subset I$ with \(i_m \to \infty\) such that \(\mu_{i_{m}}(A_{m}) > 1 - \frac{1}{m}\) for each \(m\in \mathbb{N}\) and
		\[
		\lim_{m\to\infty} \max_{t\in A_{m}} \rho(T_{t}x,0)= 0,
		\quad \forall\, x \in X_{0}.
		\]
		
		\item There exist $\varepsilon > 0$, a sequence $(y_k)_k \subset \overline{\mathrm{span}}(X_{0})$ with $y_k \to 0$, and a sequence $(j_k)_k \subset I$ with $j_k \to \infty$ such that
		\[
		\int_{\Lambda} \mathbf{1}_{(\varepsilon, \infty)}
		\bigl(\|T^{t} y_k\|_{\beta}\bigr)\, d\mu_{j_k}(t)
		\geq 1 - \frac{1}{k},
		\quad \forall\, k \in \mathbb{N}.
		\]
	\end{enumerate}
	If \(X_{0}\) can be taken dense in \(X\), we say that $(T_t)_{t\in\Lambda}$ satisfies the \emph{densely $(\mu_i)$-distributional chaos criterion}.
\end{definition}

\begin{theorem}
	$(T_{t})_{t\in \Lambda}$ is $(\mu_i)$-distributionally chaotic if and only if it satisfies the $(\mu_i)$-distributional chaos criterion.
\end{theorem}

\begin{proof}
	Assume first that $(T_{t})_{t\in \Lambda}$ is $(\mu_i)$-distributionally chaotic. Then it admits a $\beta$-distributionally \((\mu_i)\)-irregular vector $x\in X$. Taking $X_0:=\{x\}$, it follows immediately that the $(\mu_i)$-distributional chaos criterion is satisfied.
	
	Conversely, suppose that $(T_{t})_{t\in \Lambda}$ satisfies the $(\mu_i)$-distributional chaos criterion, and let $X_0\subset X$ be as in Definition~\ref{criterio-distrib}. Set
	\[
	Z:=\overline{\mathrm{span}}(X_0),
	\]
	endowed with the induced metric $\mathrm{D}$, so that $(Z,\mathrm{D})$ is an $F$-space. For each $t\in\Lambda$, define $S_t:=T_t|_Z$.
	
	By condition~(1), for every $z\in \mathrm{span}(X_0)$,
	\[
	\lim_{m\to\infty}\max_{t\in A_m}\rho(S_t z,0)=0.
	\]
	By Theorem~\ref{prod-dis-near 0}, it follows that, for each $n\in\mathbb{N}$, the set of $(\mu_i)$-distributionally near-to-zero vectors for $(\bigoplus_{k=1}^n S_t)_{t\in\Lambda}$ is residual in $Z^n$.
	
	On the other hand, by condition~(2) and Proposition~\ref{dist-unbounded}, the set of $\beta$-distributionally $(\mu_i)$-unbounded vectors for $(S_t)_{t\in\Lambda}$ is residual in $Z$. Consequently, for each $n\in\mathbb{N}$, the family $(\bigoplus_{k=1}^n S_t)_{t\in\Lambda}$ is $(\mu_i)$-distributionally chaotic.
	
	Therefore, for each $n\in\mathbb{N}$, the set
	\[
	\mathcal{F}_{(\mu_i)}\Psi_{\beta}^{\mathrm{LY}}\big((\bigoplus_{k=1}^{n} S_t)_{t\in\Lambda}\big)
	\]
	is residual in $Z^n$. Thus, $(S_t)_{t\in\Lambda}$ admits a $\beta$-distributionally \((\mu_i)\)-irregular vector. By Proposition~\ref{equi dis-cha dist-irr}, $(T_{t})_{t\in \Lambda}$ is $(\mu_i)$-distributionally chaotic.
\end{proof}

\begin{theorem}\label{prod-dense-dist}
	Let \(X\) be a separable space. The family $(T_{t})_{t\in \Lambda}$ satisfies the densely $(\mu_i)$-distributional chaos criterion if and only if, for every \(n\in \mathbb{N}\), the direct sum \(\left(\bigoplus_{k=1}^{n} T_{t}\right)_{t\in \Lambda}\) is densely $(\mu_i)$-distributionally chaotic.
\end{theorem}

\begin{proof}
	Assume that $(T_{t})_{t\in \Lambda}$ satisfies the densely $(\mu_i)$-distributional chaos criterion. Arguing as in the proof of the previous result, and using that $X_0$ is dense in $X$ together with the separability of $X$, it follows from Theorem~\ref{dense-mu_i-distri} that, for every $n \in \mathbb{N}$, the direct sum $\left(\bigoplus_{k=1}^{n} T_{t}\right)_{t\in \Lambda}$ is densely $(\mu_i)$-distributionally chaotic.
	
	Conversely, the result follows from Theorem~\ref{prod-dis-near 0}.
\end{proof}

\begin{proposition}\label{beta-desi} For \(\delta>0\), we have
	\[
	\mu_{i}(\{t\in \Lambda: \Vert{T_{t}x\Vert}_{\beta}\geq \delta\})\leq \frac{1}{\delta} \int_{\Lambda} \Vert{T_{t}x\Vert}_{\beta}\ d\mu_{i}(t)
	\]	
\end{proposition}



\section{Abstract criteria for linear structures}

The main objective of this section is to establish two abstract criteria for the existence of large linear structures, namely Lemma~\ref{dense-citerio} and Lemma~\ref{space-cri}, which will play a central role in the subsequent sections.

For the reader’s convenience, we briefly recall the relevant notions. Let $X$ be an infinite-dimensional $F$-space. A subset $A \subset X$ is said to be \emph{dense-lineable} if $A \cup \{0\}$ contains a dense linear subspace of $X$, and \emph{spaceable} if $A \cup \{0\}$ contains an infinite-dimensional closed subspace of $X$. For further aspects of these and related notions, we refer to \cite{aron2016lineability, bernal2014lineability}.

The formulation and proof of Lemma~\ref{dense-citerio}, concerning dense-lineability, are inspired by \cite[Theorem 3.36]{JiLi25}. Related adaptations of this technique have also been used to establish sufficient conditions for the existence of dense distributionally irregular manifolds for \(C_{0}\)-semigroups on complex sectors; see \cite[Theorem 3.16]{Jiang2025}.

On the other hand, Lemma~\ref{space-cri} combines structural aspects of Lemma~\ref{dense-citerio} with ideas originating from the spaceability techniques developed in \cite{GoLeMo, Lo24, SaSt}.

\subsection{A dense-lineability criterion in \(F\)-spaces}

Let $X$ be a separable infinite-dimensional $F$-space, and let 
$\mathcal{S} \subset \mathcal{P}^{\mathbb{N}}$ be an abstract class of sequences with parameter set $\mathcal{P}$. 
We assume that $\mathcal{S}$ is stable under extraction of subsequences, that is, if 
$(A_k)_{k \in \mathbb{N}} \in \mathcal{S}$ and $(k_j)_{j \in \mathbb{N}}$ is a strictly increasing sequence of integers, 
then $(A_{k_j})_{j \in \mathbb{N}} \in \mathcal{S}$.

Consider three mappings
\begin{align*}
	f,g,h \colon X \times \mathcal{P} \longrightarrow [0,\infty).
\end{align*}
Assume that the following conditions are satisfied:
\begin{itemize}
	\item For every \(\lambda \in \mathbb{K}\) and every \(x \in X\),
	\[
	f(\lambda x,\cdot)=|\lambda|\,f(x,\cdot), 
	\qquad
	g(\lambda x,\cdot)=|\lambda|\,g(x,\cdot),
	\qquad
	h(\lambda x,\cdot)\leq (|\lambda|+1)\,h(x,\cdot).
	\]
	\item The mappings \(g(x,\cdot)\) and \(h(x,\cdot)\) satisfy the triangle inequality, that is,
	\[
	g(x+y,\cdot)\leq g(x,\cdot)+g(y,\cdot)
	\quad\text{and}\quad
	h(x+y,\cdot)\leq h(x,\cdot)+h(y,\cdot)
	\]
	for all \(x,y\in X\).
	\item For all \(x,y\in X\),
	\[
	f(x+y,\cdot)\geq f(x,\cdot)-g(y,\cdot).
	\]
\end{itemize}
Define
\begin{align*}
	G:=\bigl\{x\in X:\exists (B_{k})_{k}\in \mathcal{S}\ \text{such that}\ \lim_{k\to\infty} f(x,B_{k})=\infty\bigr\}.
\end{align*}
For each \((A_{k})_{k}\in \mathcal{S}\), define
\begin{align*}
	P(A_{k})&:=\{x\in X:\liminf_{k\to\infty} g(x,A_{k})=0\},\\
	Q(A_{k})&:=\{x\in X:\liminf_{k\to\infty} h(x,A_{k})=0\}.
\end{align*}

\begin{lemma}\label{dense-citerio}
	Assume that \(G\) is residual and that, for each \((A_{k})_{k}\in \mathcal{S}\), the sets \(P(A_{k})\) and \(Q(A_{k})\) are residual in \(X\).
	Then there exists a dense manifold \(Z\subset X\) such that for every \(z\in Z\setminus\{0\}\) there exist two sequences \((\Theta_{k})_{k}\) and \((\Lambda_{k})_{k}\) in \(\mathcal{S}\) satisfying
	\begin{align}\label{req-asyn}
		\lim_{k\to\infty} h(z,\Theta_{k})=0
		\quad\text{and}\quad
		\lim_{k\to\infty} f(z,\Lambda_{k})=\infty.
	\end{align}
\end{lemma}

\begin{proof}
	Let $\{y_\ell\}_{\ell\in\mathbb{N}}$ be a countable dense subset of $X$. We shall inductively construct a sequence
	$\{x_\ell\}_{\ell\in\mathbb{N}}\subset X$ such that
	\[
	\mathrm{D}(x_\ell,y_\ell)<\frac{1}{\ell}\quad\text{for all }\ell\in\mathbb{N},
	\]
	and such that every nonzero vector in $\text{span}\{x_\ell:\ell\in\mathbb{N}\}$ satisfies \eqref{req-asyn}.
	The density of this span then follows immediately.
	
	Fix an arbitrary sequence $(A_k)_k\in\mathcal{S}$. Since $G$, $P(A_k)$ and $Q(A_k)$ are residual, we may choose
	\[
	x_1\in G\cap P(A_k)\cap Q(A_k)
	\quad\text{with}\quad
	\mathrm{D}(x_1,y_1)<1.
	\]
	Hence, there exist sequences $(A_k^{(1,1)})_k$, $(B_k^{(1)})_k$ and $(C_k^{(1)})_k$ in $\mathcal{S}$, where
	$(A_k^{(1,1)})_k$ and $(C_k^{(1)})_k$ are subsequences of $(A_k)_k$, such that
	\[
	\lim_{k\to\infty} g(x_1,A_k^{(1,1)})=0,\;\;
	\lim_{k\to\infty} f(x_1,B_k^{(1)})=\infty,
	\;\;\text{and}\\;\;
	\lim_{k\to\infty} h(x_1,C_k^{(1)})=0.
	\]
	By homogeneity, these limits hold for any nonzero scalar multiple of $x_1$.
	
	Assume now that, for some $n\ge1$, we have constructed vectors $x_1,\dots,x_n\in X$ and sequences
	\[
	(A_k^{(\ell,j)})_k,\quad (B_k^{(\ell)})_k,\quad (C_k^{(\ell)})_k \in \mathcal{S},
	\qquad 1\le j\le \ell\le n,
	\]
	satisfying the following properties:
	\begin{enumerate}
		\item $\mathrm{D}(x_\ell,y_\ell)<\frac{1}{\ell}$ for $1\le \ell\le n$;
		\item for $2\le \ell\le n$, $(C_k^{(\ell)})_k$ is a subsequence of $(C_k^{(\ell-1)})_k$;
		\item for every $1\le \ell\le n$,
		\[
		\lim_{k\to\infty} h(x_\ell,C_k^{(n)})=0;
		\]
		\item for $2\le \ell\le n$ and $1\le j\le \ell-1$,
		$(A_k^{(\ell,j)})_k$ is a subsequence of $(A_k^{(\ell-1,j)})_k$, and
		$(A_k^{(\ell,\ell)})_k$ is a subsequence of $(B_k^{(\ell-1)})_k$;
		\item for all $1\le j\le \ell\le n$,
		\[
		\lim_{k\to\infty} g(x_\ell,A_k^{(\ell,j)})=0
		\quad\text{and}\quad
		\lim_{k\to\infty} f(x_\ell,B_k^{(\ell)})=\infty;
		\]
		\item every nonzero vector in $\operatorname{span}\{x_1,\dots,x_n\}$ satisfies \eqref{req-asyn}.
	\end{enumerate}
	
	Define
	\[
	R_{n+1}:=\bigcap_{j=1}^{n} P(A_k^{(n,j)})\cap P(B_k^{(n)})\cap G\cap Q(C_k^{(n)}),
	\]
	which is residual in $X$. We may thus choose
	$x_{n+1}\in R_{n+1}$ such that
	\[
	\mathrm{D}(x_{n+1},y_{n+1})<\frac{1}{n+1}.
	\]
	By construction, for each $1\le j\le n$ there exists a subsequence
	$(A_k^{(n+1,j)})_k$ of $(A_k^{(n,j)})_k$ such that
	\[
	\lim_{k\to\infty} g(x_{n+1},A_k^{(n+1,j)})=0.
	\]
	Moreover, there exist a subsequence $(A_k^{(n+1,n+1)})_k$ of $(B_k^{(n)})_k$,
	a sequence $(B_k^{(n+1)})_k$ in $\mathcal{S}$, and a sequence $(C_k^{(n+1)})_k$ in $\mathcal{S}$ such that
	\begin{itemize}
		\item \(\displaystyle{\lim_{k\to\infty} g(x_{n+1},A_k^{(n+1,n+1)})=0}\),
		\item \(\displaystyle{\lim_{k\to\infty} f(x_{n+1},B_k^{(n+1)})=\infty}\), and
		\item \(\displaystyle{\lim_{k\to\infty} h(x_{n+1},C_k^{(n+1)})=0}\).
	\end{itemize}
	Let $0\neq z=\sum_{\ell=1}^{n+1}\alpha_\ell x_\ell\in\text{span}\{x_1,\dots,x_{n+1}\}$. Then,
	by the triangle inequality and the growth condition on $h$,
	\[
	h(z,C_k^{(n+1)})
	\le \sum_{\ell=1}^{n+1}(|\alpha_\ell|+1)h(x_\ell,C_k^{(n+1)})
	\xrightarrow[k\to\infty]{}0.
	\]
	Let $j=\min\{\ell:\alpha_\ell\neq0\}$. If $j=n+1$, then
	\[
	f(z,B_k^{(n+1)})=|\alpha_{n+1}|f(x_{n+1},B_k^{(n+1)})\xrightarrow[k\to\infty]{}\infty.
	\]
	If $j\le n$, writing $z=\sum_{\ell=j}^{n+1}\alpha_\ell x_\ell$, we obtain
	\[
	f(z,A_k^{(n+1,j)})
	\ge |\alpha_j|f(x_j,A_k^{(n+1,j)})
	-\sum_{\ell=j+1}^{n+1}|\alpha_\ell|\,g(x_\ell,A_k^{(n+1,j)}).
	\]
	By the inductive construction and item~(5), the right-hand side diverges to $+\infty$ as $k\to\infty$.
	Therefore, every nonzero vector in $\text{span}\{x_1,\dots,x_{n+1}\}$ satisfies \eqref{req-asyn}.
	
	By induction, the same holds for every nonzero vector in
	$\text{span}\{x_\ell:\ell\in\mathbb{N}\}$, which completes the proof.
\end{proof}

\subsection{A spaceability criterion in separable Banach spaces}

\begin{lemma}\label{space-cri}
	Let $X$ be a separable infinite-dimensional real or complex Banach space. 
	Assume the hypotheses of Lemma~\ref{dense-citerio} concerning 
	$\mathcal{S}$, $\mathcal{P}$, the maps $f,g$, and the sets $G$, $P(A_k)$ and $Q(A_k)$.
	
	Suppose moreover that there exists a surjective map 
	$J:\mathcal{P}\to \mathbb{N}$ and a decreasing sequence 
	$(E_n)_{n\in\mathbb{N}}$ of infinite-dimensional closed subspaces of $X$ such that
	\[
	g(x,A)\le M\|x\|
	\quad 
	\text{for all } x\in E_n,\ 
	\text{for all } A\in\mathcal{P} \text{ with } J(A)= n,
	\quad n\in\mathbb{N},
	\]
	for some constant $M>0$, and that for each $n\in\mathbb{N}$ and each $\delta>0$ 
	there exists $\varepsilon>0$ such that
	\begin{align}\label{condi J-g}
		\|y\|<\varepsilon 
		\quad \Longrightarrow \quad 
		g(y,A)<\delta
		\quad 
		\text{for all } A\in\mathcal{P} \text{ with } J(A)= n.
	\end{align}
	Then there exist an infinite-dimensional closed subspace 
	$F\subset X$ and a sequence $(\Theta_n)_{n\in\mathbb{N}}\in\mathcal{S}$  such that
	\[
	\lim_{n\to\infty} g(x,\Theta_n)=0
	\quad \text{for every } x\in F,
	\]
	and for every $x\in F\setminus\{0\}$ there exists a sequence 
	$(\Lambda_k)\in\mathcal{S}$ satisfying
	\[
	\lim_{k\to\infty} f(x,\Lambda_k)=\infty.
	\]
\end{lemma}

Let $(E_{n})_{n}$ be the collection of subspaces appearing in the second condition of the hypothesis of the previous theorem. By Mazur’s construction \cite[Lemma C.1.1]{BaMa}, there exists a normalized basic sequence $(e_n)_{n \in \mathbb{N}}$ with $e_n \in E_n$, which constitutes a Schauder basis for $E := \overline{\mathrm{span}}\{e_n : n \in \mathbb{N}\}$. The same property holds for any subsequence of a normalized basic sequence. Denote by $(e_n^*)_{n} \subset E^*$ the corresponding sequence of coordinate functionals on $E$, characterized by $\langle e_m^*, \sum_k \alpha_k e_k \rangle = \alpha_m$ for each $m \in \mathbb{N}$. Moreover, since $(e_n)_{n}$ is normalized, it follows that $\sup_n \|e_n^*\| < \infty$ \cite{Me}, and we define $C := 1 + \sup_{m} \|e_m^*\|$.

\begin{proof}
	Put $\omega_{1}=1$ and adopt the notation of the sets $P(A_k)$ and $\mathrm{G}$ as in the proof of Lemma~\ref{dense-lineable}. 
	Choose
	\[
	p_{\omega_{1}} \in P(A_{k})\cap \mathrm{G}
	\quad \text{such that} \quad 
	\|p_{\omega_{1}}-e_{\omega_{1}}\|<2^{-(1+1)}C^{-1}.
	\]
	By definition of $P(A_{k})$ and $\mathrm{G}$, there exist subsequences
	$(A_{k}^{(1,1)})_{k}$ of $(A_{k})_{k}$ and $(B_{k}^{(1)})_{k}$ in $\mathcal{S}$ such that
	\[
	\lim_{k\to\infty}g(p_{\omega_{1}},A_{k}^{(1,1)})=0,
	\qquad
	\lim_{k\to\infty}f(p_{\omega_{1}},B_{k}^{(1)})=\infty.
	\]
	Hence, we may select \(\Theta_{1}\)
	from the sequence \((A_{k}^{(1,1)})_{k}\) and $\Lambda_{1,1}$ from the sequence $(B_{k}^{(1)})_{k}$ such that
	\[
	g(p_{\omega_{1}}, \Theta_{1})<2^{-(1+1)},
	\qquad
	f(p_{\omega_{1}},\Lambda_{1,1})>1.
	\]
	
	Next, choose $\omega_{2}>\max\{J(\Theta_{1}),J(\Lambda_{1,1})\}$ and fix $e_{\omega_{2}}\in E$.
	By condition (\ref{condi J-g}), there exists $\varepsilon_{2}>0$ such that
	\[
	g(y, A)<2^{-(2+\zeta)}
	\quad  \text{for all } J(A)=\zeta\,; 1\le \zeta\le \omega_{2} 
	\]
	whenever $y\in X$ satisfies $\|y\|<\varepsilon_{2}$.
	
	We claim that there exist a sequence $(p_{\omega_{\ell}})_{\ell\in\mathbb{N}}\subset X$,
	a strictly increasing sequence of positive integers $(\omega_{\ell})_{\ell\in\mathbb{N}}$, a sequence $(\Theta_{\ell})_{\ell\in \mathbb{N}}\in \mathcal{S}$,
	and, for each $\ell\in\mathbb{N}$, a sequence
	$(\Lambda_{k,\ell})_{k\in\mathbb{N}}\in \mathcal{S}$ such that the following properties hold:
	\begin{enumerate}
		\item $\|p_{\omega_{\ell}}-e_{\omega_{\ell}}\|<2^{-(\ell+1)}C^{-1}$ for every $\ell\in\mathbb{N}$;
		\item $g(p_{\omega_{\ell}}, \Theta_{n})<2^{-(\ell+n)}$ for all $1\le \ell\le n$, $n\in\mathbb{N}$;
		\item $g(p_{\omega_{\ell}}-e_{\omega_{\ell}}, A)<2^{-(\ell+\zeta)}$
		for all \( J(A)= \zeta\,; 1\le \zeta\le \omega_{n}\).
		\item for each $j\in\mathbb{N}$,
		\[
		f(p_{\omega_{j}},\Lambda_{k,j})>k \quad \text{for all } k\in\mathbb{N},
		\]
		and
		\[
		g(p_{\omega_{\ell}},\Lambda_{k,j})<2^{-(\ell+k)}
		\quad \text{for } j\le \ell\le k+j-1,\; k\in\mathbb{N}.
		\]
	\end{enumerate}
	The proof of the above statement proceeds by induction. Assume that we have constructed $(p_{\omega_{\ell}})_{\ell=1}^{n}\subset X$, $(\omega_{\ell})_{\ell=1}^{n}$,  $(\Theta_{\ell})_{\ell=1}^{n}$, $(\Lambda_{k,j})_{\substack{k+j\leq n+1\\ j,k\geq 1}}$, and $(\varepsilon_{\ell})_{\ell=2}^{n}$ such that 
	\begin{itemize}
		\item $\displaystyle{\|p_{\omega_{\ell}}-e_{\omega_{\ell}}\|<\min\{2^{-(\ell+1)}C^{-1}, \varepsilon_{\ell}\}}$ for $1\leq \ell\leq n$;
		\item $g(p_{\omega_{\ell}}, \Theta_{n})<2^{-(\ell+n)}$ for $1\leq \ell\leq n$;
		\item $g(y, A)<2^{-(k+\zeta)}$, for each \(J(A)=\zeta\), $1\leq \zeta \leq \omega_{k}$, $\|y\|<\varepsilon_{k}$,  $1\leq k\leq n$;
		\item $f(p_{\omega_{j}}, \Lambda_{k,j}) >k$, $1\leq j\leq n$;
		\item $g(p_{\omega_{\ell}}, \Lambda_{k,j})<2^{-(\ell+k)}$, $1\leq j<\ell\leq n$.
	\end{itemize}
	Such elements of these sequences are chosen from the sequences
	$(A_i^{(\ell,j)})_i$ and $(B_i^{(\ell)})_i$, with $1 \leq j \leq \ell \leq n$, as in the proof of Lemma~\ref{dense-citerio}.
	
	Let
	\[
	\omega_{n+1}>\max \{\omega_{n}, J(\Theta_{n}), J(\Lambda_{n,1}), \ldots, J(\Lambda_{1,n})\}.
	\]
	By the condition (\ref{condi J-g}), there exists $\varepsilon_{n+1}>0$ such that
	\[
	g(y,A)< 2^{-(\zeta +n+1)}, \quad J(A)=\zeta,\;
	1\leq \zeta\leq \omega_{n+1},\ \|y\|<\varepsilon_{n+1}.
	\]
	Now, for $e_{\omega_{n+1}}\in E$, we may choose
	\[
	p_{\omega_{n+1}}\in \bigcap_{j=1}^{n} P(A_i^{(n,j)})\cap P(B_i^{(n)})\cap \mathrm{G}
	\]
	such that
	\[
	\|p_{\omega_{n+1}}-e_{\omega_{n+1}}\|
	<\min\{2^{-(n+2)}C^{-1}, \varepsilon_{n+1}\}.
	\]
	Then, there exist subsequences $(A_{i}^{(n+1,j)})_{i}$ of $(A_{i}^{(n,j)})_{i}$ for $1\leq j\leq n$,
	a subsequence $(A_{i}^{(n+1,n+1)})_{i}$ of $(B_{i}^{(n)})_{i}$,
	and a sequence $(B_{i}^{(n+1)})_{i}$ such that
	\[
	\lim_{i\to\infty}g(p_{\omega_{n+1}}, A_{i}^{(n+1,\ell)})=0,
	\quad 1\leq \ell\leq n+1,
	\]
	and
	\[
	\lim_{i\to\infty}f(p_{\omega_{n+1}}, B_{i}^{(n+1)})=\infty.
	\]
	From the subsequence $(A_i^{(n+1,1)})_i$, which is itself a subsequence of $(A_i^{(1,1)})_i$, 
	we choose $\Theta_{n+1}$ with index strictly larger than that of $\Theta_n$. 
	This yields a strictly increasing sequence of indices 
	$m_1 < m_2 < \cdots$ such that 
	$\Theta_\ell = A_{m_\ell}^{(1,1)}$ for every $\ell \le n+1$. Moreover, choosing the index sufficiently large, we may ensure that
	\[
	g(p_{\omega_{\ell}}, \Theta_{n+1})
	<2^{-(\ell+n+1)},\quad 1\leq \ell\leq n+1.
	\]
	For each $1 \le j \le n$, 
	arguing as in the choice of $\Theta_{n+1}$, 
	we choose $\Lambda_{n+2-j,j}$ from $(A_i^{(n+1,j+1)})_i$ 
	with index sufficiently large to ensure that
	\[
	f(p_{\omega_j}, \Lambda_{n+2-j,j}) > n+2-j.
	\]
	and
	\[
	g(p_{\omega_{\ell}},\Lambda_{n+2-j,j})
	<2^{-(\ell+n+2-j)}, \quad 1\leq j<\ell\leq n+1.
	\]
	Additionally, we choose $\Lambda_{1,n+1}$ such that
	\[
	f(p_{\omega_{n+1}}, \Lambda_{1,n+1})>1.
	\]
	This completes the proof of the claim.
	
	Since $(A_i^{(1,1)})_i \in \mathcal{S}$ and $\mathcal{S}$ is stable under extraction of subsequences, 
	it follows that $(\Theta_n)_{n\in\mathbb{N}} \in \mathcal{S}$. 
	Likewise, for each fixed $j \in \mathbb{N}$, the sequence $(\Lambda_{k,j})_{k}$ 
	is obtained as a subsequence of a sequence $(B_i^{(j)})_i \in \mathcal{S}$; hence,
	\[
	(\Lambda_{k,j})_{k \in \mathbb{N}} \in \mathcal{S}.
	\]
	
	\medskip
	
	Note that, by condition (1),
	\[
	\sum_{n \in \mathbb{N}} \|e^*_{\omega_n}\| \|p_{\omega_n} - e_{\omega_n}\|
	= \sum_{n \in \mathbb{N}} \|e^*_{\omega_n}\| \|q_{\omega_n}\|
	\leq \sum_{n \in \mathbb{N}} 2^{-(n+1)} < 1.
	\]
	Hence, by \cite[Lemma 10.6]{GrPe}, the sequences $(p_{\omega_n})_n$ and
	$(e_{\omega_n})_n$ are equivalent basic sequences.
	Let
	\[
	F := \overline{\mathrm{span}}\{p_{\omega_n} : n \in \mathbb{N}\}.
	\]
	
	\textbf{Claim.}
	\[
	\lim_{n\rightarrow \infty}g(x, \Theta_{n})=0,
	\quad \forall x\in F.
	\]
	
	Fix $n\in \mathbb{N}$ and $x\in F$. Write
	$x=\sum_{\ell} \beta_\ell p_{\omega_\ell}$, where $(\beta_\ell)_\ell\in c_0(\mathbb{N})$.
	Recall that $p_{\omega_\ell}=e_{\omega_\ell}+q_{\omega_\ell}$.
	Since $(e_{\omega_\ell})$ is a basic sequence equivalent to
	$(p_{\omega_\ell})$, the series $\sum_\ell \beta_\ell e_{\omega_\ell}$ converges.
	Moreover, by condition (1), the series $\sum_\ell \beta_\ell q_{\omega_\ell}$
	is absolutely convergent.
	
	Thus,
	\[
	x=\sum_{\ell=1}^{n}\beta_{\ell}p_{\omega_{\ell}}
	+\sum_{\ell>n}\beta_{\ell}e_{\omega_{\ell}}
	+\sum_{\ell>n}\beta_{\ell}q_{\omega_{\ell}},
	\]
	and therefore
	\begin{align*}
		g(x, \Theta_{n})
		&\leq \sum_{\ell=1}^{n}|\beta_{\ell}|
		g(p_{\omega_{\ell}}, \Theta_{n})
		+ g(\sum_{\ell>n}\beta_{\ell}e_{\omega_{\ell}}, \Theta_{n})
		+ \sum_{\ell>n}|\beta_{\ell}|
		g(q_{\omega_{\ell}}, \Theta_{n}) \\
		&\leq \frac{\|\beta\|_{\infty}}{2^{n}}
		+M \Bigl\|\sum_{\ell>n}\beta_{\ell}e_{\omega_{\ell}}\Bigr\|.
	\end{align*}
	
	Since $(e_{\omega_n})$ is a basic sequence, it follows that
	$\bigl\|\sum_{\ell>n}\beta_n e_{\omega_n}\bigr\|\to 0$ as $n\to\infty$,
	which proves the claim.
	
	\textbf{Claim.}
	\[
	\limsup f(x, \cdot)=\infty,
	\quad \forall x\in F\setminus \{0\}.
	\]
	
	Let $x\in F\setminus\{0\}$. Write
	$x=\sum_{\ell} \beta_\ell p_{\omega_\ell}$, with
	$(\beta_\ell)_\ell\in c_0(\mathbb{N})$, and set
	\[
	j=\min\{\ell\in \mathbb{N}:\beta_\ell\neq 0\}.
	\]
	Then, for each $k\in \mathbb{N}$,
	\[
	x=\beta_{j}p_{\omega_{j}}
	+\sum_{\ell=j+1}^{k+j}\beta_{\ell}p_{\omega_{\ell}}
	+\sum_{\ell>k+j}\beta_{\ell}e_{\omega_{\ell}}
	+\sum_{\ell>k+j}\beta_{\ell}q_{\omega_{\ell}}.
	\]
	To establish the claim, we consider the sequence $(\Lambda_{k,j})_{k}$.
	Observe that
	\[
	g\Bigl(\sum_{\ell=j+1}^{k+j}\beta_{\ell}p_{\omega_{\ell}}, \Lambda_{k,j}\Bigr)
	\leq
	\sum_{\ell=j+1}^{k+j}|\beta_{\ell}|
	g(p_{\omega_{\ell}}, \Lambda_{k,j})
	<
	\frac{\|\beta\|_{\infty}}{2^{k}}.
	\]
	Moreover, since \(J(\Lambda_{k,j})<\omega_{k+j+1}\)
	\begin{align*}
		g\Bigl(\sum_{\ell>k+j}\beta_{\ell}e_{\omega_{\ell}},\Lambda_{k,j}\Bigr)
		&\leq
		M
		\cdot
		\Bigl\|\sum_{\ell>k+j}\beta_{\ell}e_{\omega_{\ell}}\Bigr\| \xrightarrow[k\rightarrow\infty]{}0
	\end{align*}
	since $(e_{\omega_\ell})$ is a basic sequence.
	
	Finally, by condition~(\ref{condi J-g}),
	\begin{align*}
		g\Bigl(\sum_{\ell>k+j}\beta_{\ell}q_{\omega_{\ell}}, \Lambda_{k,j}\Bigr)
		&\leq
		\sum_{\ell>k+j}|\beta_{\ell}|
		g(q_{\omega_{\ell}}, \Lambda_{k,j}) \\
		&\leq
		\|\beta\|_{\infty}
		\sum_{\ell>k+j}\frac{1}{2^{\ell+k}}.
	\end{align*}
	Therefore, combining the above estimate with the previous bounds and using the triangle inequality for
	\(x-\beta_j p_{\omega_j}\), we obtain
	\begin{align*}
		f(x,\Lambda_{k,j})
		&\geq
		|\beta_{j}|
		f(p_{\omega_{j}}, \Lambda_{k,j})
		-
		g(x-\beta_{j}p_{\omega_{j}}, \Lambda_{k,j}) \\
		&\geq
		|\beta_{j}|\,k
		-
		g(x-\beta_{j}p_{\omega_{j}}, \Lambda_{k,j}) 
		\xrightarrow[k\to\infty]{}\infty.
	\end{align*}
	This completes the proof.
\end{proof}




\section{Dense-Lineable Structures in Irregularity}

We say that an operator \(T\) on a separable \(F\)-space \(X\) is \emph{hypercyclic} if there exists \(x\in X\) such that the orbit \(\{T^{n}x: n\in \mathbb{N}\}\) is dense in \(X\). The set of all such vectors is denoted by \(\mathrm{HC}(T)\). It is well known that, whenever \(T\) is hypercyclic, the set \(\mathrm{HC}(T)\) forms a dense \(G_{\delta}\)-subset of \(X\). Moreover, it was shown in \cite{bes1999invariant, bourdon1993invariant, He91, We03} that \(\mathrm{HC}(T)\) is always dense-lineable.

Another class of operators recently introduced and studied in linear dynamics is that of recurrent operators; see \cite{costakis2012szemeredi,costakis2014recurrent}. If \(T\) is recurrent on an \(F\)-space, then the set of recurrent vectors \(\mathrm{Rec}(T)\) is again a \(G_{\delta}\)-subset of \(X\). Nevertheless, recent results show that there exist recurrent operators on separable Banach spaces for which \(\mathrm{Rec}(T)\) fails to be dense-lineable; see \cite{lopez2025two, SaSt}.

For operators exhibiting some form of irregular behavior, several sufficient conditions ensuring the dense-lineability of the corresponding sets of irregular vectors have been established in \cite{BeBoMaPe, BeBoMuPe13, BeBo15, BeBoPe20, BeBoPeWu18, JiLi25}, among others. Nevertheless, it is still unknown whether every densely Li--Yorke chaotic operator on a separable Banach space admits a dense irregular manifold; see \cite[Problem~1]{BeBo15}.

The situation changes in the setting of sequences of operators. Indeed, A. Arbieto and M. Saavedra showed that every infinite-dimensional separable complex Banach space \(X\) supports a densely Li--Yorke chaotic sequence of bounded operators \((T_n)_n\subset\mathcal L(X)\) which admits no dense irregular manifold \cite[Theorem~4.3]{ArSa}.

The purpose of this section is to establish sufficient conditions for the existence of dense absolutely \((\mu_i)\)-irregular manifolds and dense distributionally \((\mu_i)\)-irregular manifolds for the family \((T_t)_{t\in\Lambda}\); see Theorems~\ref{dense-lineable} and Theorem \ref{dense-disti}, respectively.

\subsection{Dense lineability of absolutely \((\mu_{m})\)-irregular vectors}

\begin{definition}
	We say that the family \((T_t)_{t\in \Lambda}\subset \mathcal{L}(X,Y)\) admits a dense absolutely \((\mu_i)_{i\in I}\)-irregular manifold if there exist \(\beta\in \mathbb{N}\) and a dense linear subspace 
	\(Z\subset X\) such that, for every \(x\in Z\setminus \{0\}\),
	\[
	\liminf_{i\to\infty}
	\int_{\Lambda} \rho(T_t x,0)\, d\mu_i(t)=0,
	\qquad \text{and} \qquad
	\limsup_{i\to\infty}
	\int_{\Lambda} \|T_t x\|_{\beta}\, d\mu_i(t)=\infty.
	\]
\end{definition}

We begin with the following dichotomy, which provides a sufficient condition for the existence of a dense absolutely \((\mu_i)\)-irregular manifold.

\begin{theorem}\label{dense-lineable}
	Let $(T_t)_{t\in \Lambda} \subset \mathcal{L}(X,Y)$ with \(X\) separable. Suppose that the set
	\begin{align}\label{dense-semin-condi}
		\bigcap_{\alpha \in \mathbb{N}}\Bigl\{x\in X : \lim_{\substack{i\to\infty\\ i\in I}}
		\int_{\Lambda} \|T_t x\|_{\alpha}\, d\mu_i(t)=0 \Bigr\}
	\end{align}
	is dense in $X$. Then one of the following alternatives holds:
	\begin{enumerate}
		\item \(\displaystyle{\lim_{i\to\infty}
			\int_{\Lambda} \Vert{T_t x\Vert}_{\beta}\, d\mu_i(t)=0\;\; \forall x\in X, \forall \beta\in \mathbb{N}}\),
		
		\item
		the family $(T_t)_t$ admits a dense absolutely $(\mu_i)_{i\in I}$-irregular manifold.
	\end{enumerate}
\end{theorem}

Before proving the preceding statement, we require the following auxiliary result.

\begin{proposition}\label{closed-sub}
	Let $(T_t)_{t} \subset \mathcal{L}(X,Y)$.  
	If $(T_t)_{t}$ is absolutely $(\mu_i)_{i}$-bounded, then, for each
	$\beta\in\mathbb{N}$, the set
	\[
	A_{\beta}
	:=
	\Bigl\{ x\in X : \lim_{i\to\infty}\int_{\Lambda} \|T_t x\|_{\beta}\, d\mu_i(t)=0 \Bigr\}
	\]
	is a closed linear subspace of $X$.
\end{proposition}

\begin{proof}
	Let $z\in \overline{A_{\beta}}$ and fix $\varepsilon>0$. Since $(T_t)_t$ is
	absolutely $(\mu_i)_i$-bounded with respect to $\|\cdot\|_{\beta}$, there exists
	$r>0$ such that
	\[
	\sup_{\substack{y\in X\\ \mathrm{D}(y,0)<r}}
	\ \sup_{i\in I}
	\int_{\Lambda} \|T_t y\|_{\beta}\, d\mu_i(t)
	< \varepsilon .
	\]
	Choose $x\in A_{\beta}$ with $\mathrm{D}(z,x)<r$. Then, for every $i$,
	\[
	\int_{\Lambda} \|T_t z\|_{\beta}\, d\mu_i(t)
	\le
	\int_{\Lambda} \|T_t(z-x)\|_{\beta}\, d\mu_i(t)
	+
	\int_{\Lambda} \|T_t x\|_{\beta}\, d\mu_i(t).
	\]
	Taking the limit superior as $i\to\infty$ and using that $x\in A_{\beta}$, we obtain
	\[
	\limsup_{i\to\infty}\int_{\Lambda} \|T_t z\|_{\beta}\, d\mu_i(t)
	\le \varepsilon .
	\]
	Since $\varepsilon>0$ is arbitrary, it follows that
	\[
	\lim_{i\to\infty}\int_{\Lambda} \|T_t z\|_{\beta}\, d\mu_i(t)=0,
	\]
	and hence $z\in A_{\beta}$. Therefore $A_{\beta}$ is closed in $X$.
\end{proof}

\begin{proof}[Proof of Theorem \ref{dense-lineable}]
	Assume first that $(T_t)_t$ is absolutely $(\mu_i)$-bounded. By Proposition~\ref{closed-sub} in the condition ~\ref{dense-semin-condi}
	\[
	\lim_{i\to\infty}\int_{\Lambda} \Vert{T_t x\Vert}_{\beta}\, d\mu_i(t)=0
	\qquad \forall x\in X, \forall \beta\in \mathbb{N}.
	\]
	Assume now that $(T_t)_t$ is not absolutely $(\mu_i)$-bounded, that is, there exists \(\beta\in \mathbb{N}\) such that \((T_{t})_{t}\) is not absolutely $(\mu_i)$-bounded with respect to \(\Vert{\cdot\Vert}_{\beta}\).  By Corollary~\ref{not-abs-bounded} together with condition~\ref{compact-soporte}, 
	there exists a strictly increasing and unbounded sequence 
	\((i_n)_{n\in\mathbb N}\subset I\) such that the set
	\[
	\mathrm{G}:=
	\Bigl\{
	x\in X :
	\limsup_{n\to\infty}
	\int_{\Lambda} \|T_t x\|_{\beta}\, d\mu_{i_n}(t)=\infty
	\Bigr\}
	\]
	is residual in \(X\).
	
	Moreover, condition (\ref{dense-semin-condi}) ensures that the set
	\[
	\Bigl\{
	x\in X :
	\lim_{i\to\infty}
	\int_{\Lambda} \rho(T_t x,0)\, d\mu_i(t)=0
	\Bigr\}
	\]
	is dense in \(X\).
	
	Now set \(\mathcal P := \{i_n : n\in\mathbb N\}\) and define
	\[
	\mathcal S :=
	\left\{
	(m_k)_{k\in\mathbb N}\subset \mathcal P :
	\text{ strictly increasing}
	\right\}.
	\]
	Define the maps \(f,g,h : X\times \mathcal P \to [0,\infty) \) by
	\[
	f(x,m)=g(x,m):=
	\int_{\Lambda} \|T_t x\|_{\beta}\, d\mu_m(t),
	\quad
	h(x,m):=
	\int_{\Lambda} \rho(T_t x,0)\, d\mu_m(t).
	\]
	
	For each \((m_k)_k\in\mathcal{S}\), the sets
	\[
	P(m_k):=\left\{
	x\in X:\;
	\liminf_{k\to\infty}\int_{\Lambda}\|T_t x\|_{\beta}\,d\mu_{m_{k}}(t)=0
	\right\}
	\]
	and
	\[
	Q(m_k):=\left\{
	x\in X:\;
	\liminf_{k\to\infty}\int_{\Lambda}\rho(T_t x,0)\,d\mu_{m_{k}}(t)=0
	\right\}
	\]
	are residual in \(X\).

	By Lemma~\ref{dense-citerio}, there exists a dense subspace 
	\(Z\subset X\) such that, for every \(x\in Z\setminus \{0\}\),
	\[
	\liminf_{i\to\infty}
	\int_{\Lambda} \rho(T_t x,0)\, d\mu_i(t)=0,
	\quad	\text{and}\quad
	\limsup_{i\to\infty}
	\int_{\Lambda} \|T_t x\|_{\beta}\, d\mu_i(t)=\infty.
	\]
	Therefore, \((T_t)_t\) admits a dense absolutely 
	\((\mu_i)_{i\in I}\)-irregular manifold.
\end{proof}

\begin{corollary}
	Let $(\mu_m)_{m\in\mathbb{N}}$ satisfy \textup{(M1)}. Let $B_{\omega}$ be a unilateral weighted backward shift on a Fréchet or Banach space $X$ admitting a basis $(e_n)_{n\in\mathbb{N}}$. Then exactly one of the following alternatives holds:
	\begin{enumerate}
		\item For each $x\in X$ and each \(\beta\in \mathbb{N}\)
		\[
		\lim_{m\to\infty} \int_{\mathbb{N}} \Vert{B_{\omega}^{t}x\Vert}_{\beta} \, d\mu_m(t)=0;
		\]
		\item The operator $B_{\omega}$ admits a dense absolutely $(\mu_m)$-irregular manifold.
	\end{enumerate}
\end{corollary}



\subsection{Dense Lineability of distributionally \((\mu_m)\)-irregular Vectors}

\begin{definition}
	We say that the family \((T_t)_{t\in \Lambda}\subset \mathcal{L}(X,Y)\) admits a dense distributionally \((\mu_i)_{i\in I}\)-irregular manifold if there exist \(\beta\in \mathbb{N}\) and a dense linear subspace 
	\(Z\subset X\) such that
	\[
	Z\setminus \{0\}\subset 
	\mathcal{F}_{(\mu_i)}\Psi_{\beta}^{\mathrm{LY}}
	((T_t)_{t\in \Lambda}).
	\]
\end{definition}

We now establish the corresponding criterion for the existence of dense distributionally \((\mu_i)\)-irregular manifolds.

\begin{theorem}\label{dense-disti}
	Let \((T_{t})_{t\in \Lambda}\subset \mathcal{L}(X,Y)\), where \(X\) is separable. Suppose that there exists a family \((C_{i})_{i\in I}\) of compact subsets of \(\Lambda\) such that \(\mu_{i}(C_{i})\to 1\) as \(i\to\infty\), and that the set
	\begin{align}\label{subspace-distri}
		\bigcap_{\alpha \in \mathbb{N}}\Bigl\{x\in X : \lim_{i\to\infty}
		\max_{t\in C_{i}} \|T_t x\|_{\alpha}=0 \Bigr\}
	\end{align}
	is dense in \(X\). Then the following statements are equivalent:
	\begin{enumerate}
		\item \((T_t)_{t\in \Lambda}\) admits a \(\beta\)-distributionally \((\mu_i)\)-unbounded vector for some \(\beta\in\mathbb{N}\);
		\item \((T_t)_{t\in \Lambda}\) admits a dense distributionally \((\mu_i)\)-irregular manifold.
	\end{enumerate}
\end{theorem}

\begin{proof}
	Assume that \((T_t)_{t\in \Lambda}\) admits a \(\beta\)-distributionally \((\mu_i)\)-unbounded vector \(y\in X\). By \eqref{beta-unbo}, there exist a sequence of compact subsets \((D_k)_{k\in\mathbb N}\subset \Lambda\), a sequence \((i_k)_k\subset I\) with \(i_k\to\infty\), and \(\beta\in\mathbb N\) such that
	\[
	\mu_{i_k}(D_k)\to 1
	\quad\text{and}\quad
	\lim_{k\to\infty}\min_{t\in D_k}\|T_t y\|_{\beta}=\infty.
	\]	
	For each \(k\in \mathbb{N}\), set \(L_k:=D_k\cap C_{i_k}\). Then \(L_k\) is a compact subset of \(\Lambda\) and \(\mu_{i_k}(L_k)\to 1\) as \(k\to\infty\).
	
	Let \(z\) be as in \eqref{subspace-distri}. A straightforward estimate yields
	\begin{align*}
		\min_{t\in L_{k}}\|T_{t}(y+z)\|_{\beta}
		&\geq \min_{t\in L_{k}}\|T_{t}y\|_{\beta} - \max_{t\in L_{k}}\|T_{t}z\|_{\beta}\\
		&\geq \min_{t\in D_{k}}\|T_{t}y\|_{\beta} - \max_{t\in C_{i_{k}}}\|T_{t}z\|_{\beta}
		\longrightarrow \infty.
	\end{align*}	
	Consequently, the set
	\[
	G:=\left\{
	x\in X:\;
	\limsup_{k\to\infty}\min_{t\in L_k}\|T_t x\|_{\beta}=\infty
	\right\}
	\]
	is residual in \(X\).
	
	Let \(\mathcal{P}\) denote the family of nonempty compact subsets of \(\Lambda\), and set
	\[
	\mathcal{S}:=\{(A_k)_k : (A_k)_k \text{ is a subsequence of } (L_k)_k\}.
	\]
	Define \(f,g,h : X\times \mathcal{P} \to [0,\infty)\) by
	\[
	f(x,A)=\min_{t\in A}\|T_t x\|_{\beta},\quad
	g(x,A)=\max_{t\in A}\|T_t x\|_{\beta},\quad
	h(x,A)=\max_{t\in A}\rho(T_t x,0).
	\]
	
	By \eqref{subspace-distri} and since \(L_k\subset C_{i_k}\), it follows that for each \((A_k)_k\in\mathcal{S}\), the sets
	\[
	P(A_k):=\left\{
	x\in X:\;
	\liminf_{k\to\infty}\max_{t\in A_k}\|T_t x\|_{\beta}=0
	\right\}
	\]
	and
	\[
	Q(A_k):=\left\{
	x\in X:\;
	\liminf_{k\to\infty}\max_{t\in A_k}\rho(T_t x,0)=0
	\right\}
	\]
	are residual in \(X\).
	
	By Lemma~\ref{dense-citerio}, there exists a dense subspace \(Z\subset X\) such that for every \(x\in Z\setminus\{0\}\) there exist sequences of compact sets \((A_m)_m\), \((B_m)_m\subset \Lambda\) and sequences \((i_m)_m\), \((j_m)_m\subset I\) with \(i_m,j_m\to\infty\) satisfying
	\[
	\mu_{i_m}(A_m)\to 1,\qquad
	\mu_{j_m}(B_m)\to 1,
	\]
	and
	\[
	\lim_{m\to\infty}\max_{t\in A_m}\rho(T_t x,0)=0,\qquad
	\lim_{m\to\infty}\min_{t\in B_m}\|T_t x\|_{\beta}=\infty.
	\]
	
	In particular, for some \(\beta\in\mathbb N\),
	\[
	Z\setminus \{0\}\subset \mathcal{F}_{(\mu_i)}\Psi_{\beta}^{\mathrm{LY}}((T_t)_{t\in \Lambda}),
	\]
	which completes the proof.
\end{proof}

When $\Lambda=\mathbb{N}$ and $I=\mathbb{N}$, we say that
$(\mu_m)_{m\in\mathbb{N}}$ satisfies condition \((\mathrm{M1})\) if
\[
\lim_{m\to\infty}\mu_m(K)=0
\]
for every finite set $K\subset\mathbb{N}$. Further details concerning
this condition will be given in Section~\ref{b.t.o.s}.

\begin{theorem}
	Let \(X\) be a Fréchet sequence space in which \((e_{n})_{n\in \mathbb{N}}\) is a basis, and let \((\mu_{m})_{m\in \mathbb{N}}\) satisfy \((\mathrm{M1})\). Suppose that the unilateral backward shift \(T\) acts on \(X\). If there exists a set \(A\subset \mathbb{N}\) with \(\displaystyle{\limsup_{m\to\infty}\mu_{m}(A)=1}\) such that
	\[
	\sum_{\ell\in A} e_{\ell}
	\]
	converges in \(X\), then \(T\) admits a dense distributionally \((\mu_i)\)-irregular manifold.
\end{theorem}

\begin{proof}
	For each \(k\in\mathbb{N}\), define
	\[
	y_k:=\sum_{\substack{n\in A\\ n\ge k}}e_n.
	\]
	Since \(\sum_{\ell\in A}e_\ell\) converges in \(X\), we have \(y_k\to0\). Moreover, for every \(n\in A\) with \(n\ge k\),
	\(T^{\,n-1}y_k=e_1+\cdots\).
	
	Since the coordinate functional \(x\mapsto x_1\) is continuous, there exists \(\varepsilon>0\) such that
	\(d(x,0)<2\varepsilon\) implies \(|x_1|<1\).
	
	Choose \(\beta\in\mathbb{N}\) so that \(2^{-\beta}<\varepsilon\). Then, for every \(n\in A\) with \(n\ge k\),
	\[
	\|T^{\,n-1}y_k\|_\beta+\frac1{2^\beta}
	>d(T^{\,n-1}y_k,0)>2\varepsilon,
	\]
	and therefore
	\[
	\|T^{\,n-1}y_k\|_\beta>\varepsilon.
	\]
	
	Consequently, there exists a strictly increasing sequence of positive integers
	\((m_k)_k\) such that
	\[
	\int_{\mathbb N}\mathbf 1_{(\varepsilon,\infty)}
	\bigl(\|T^ty_k\|_\beta\bigr)\,d\mu_{m_k}(t)\longrightarrow1.
	\]
	By Proposition~\ref{dist-unbounded}, \(T\) admits a \(\beta\)-distributionally \((\mu_m)\)-unbounded vector. The conclusion now follows from Theorem~\ref{dense-disti}.
\end{proof}



\section{Spaceable Structures in Irregularity}

In recent decades, spaceability phenomena associated with sets of vectors exhibiting prescribed dynamical properties on a separable infinite-dimensional Banach space \(X\) have attracted considerable attention. Prominent examples include the set of hypercyclic vectors \(\mathrm{HC}(T)\) and the set of recurrent vectors \(\mathrm{Rec}(T)\). We refer the reader to \cite{GoLeMo,LeMo1997,LeMo2001,Lo24,Mo1996,SaSt}, as well as to \cite[Chapter~8]{BaMa} and \cite[Chapter~10]{GrPe}, for further developments in this direction.

These spaceability phenomena are often connected with structural properties of the underlying operator. Typical sufficient conditions involve the intersection of the essential spectrum of \(T\) with the closed unit disk \(\overline{\mathbb{D}}\), or the existence of an infinite-dimensional closed subspace \(E\subset X\) together with an increasing sequence \((\theta_n)_n\) of positive integers such that
\begin{align*}
	\sup_{n\in\mathbb N}\bigl\|T^{\theta_n}|_{E}\bigr\|<\infty.
\end{align*}

Motivated by these developments, we consider the question of whether the set of irregular vectors for \(T\),
\[
\Bigl\{x\in X:\,
\liminf_{n\to\infty}\|T^{n}x\|=0
\ \text{and}\
\limsup_{n\to\infty}\|T^{n}x\|=\infty
\Bigr\},
\]
is spaceable.

The aim of this section is to establish sufficient conditions on the operator \(T\) ensuring that sets of vectors exhibiting various forms of irregular behavior are spaceable, as will be shown in Theorems~\ref{condi-suf-space-irre} and Theorem \ref{condi-suf-spa-dis-irre}. The proofs rely on Lemma~\ref{space-cri}, and the form of the assumptions is inspired by \cite[Theorem~2.1]{GoLeMo}, \cite[Theorem~3.2]{Lo24}, and \cite[Theorem~4.5]{SaSt}.

We begin by recalling two spectral notions that will play a role in our arguments. Let \(T\in\mathcal L(X)\). The essential spectrum of \(T\), denoted by \(\sigma_e(T)\), consists of all \(\lambda\in\mathbb C\) such that \(T-\lambda\) fails to be a Fredholm operator. Recall that an operator \(S\in\mathcal L(X)\) is Fredholm whenever \(\operatorname{Ran}(S)\) is closed, \(\dim(\ker S)<\infty\), and \(\operatorname{codim}(\operatorname{Ran}(S))<\infty\). Similarly, the left essential spectrum of \(T\), denoted by \(\sigma_{\ell e}(T)\), consists of all \(\lambda\in\mathbb C\) such that \(T-\lambda\) fails to be a left-Fredholm operator. Here \(S\in\mathcal L(X)\) is called left-Fredholm if \(\operatorname{Ran}(S)\) is closed and \(\dim(\ker S)<\infty\).

We formulate our results in terms of the left essential spectrum rather than the essential spectrum. For hypercyclic and recurrent operators, these two notions coincide; see \cite{GoLeMo, Lo24}.

\begin{lemma}[\cite{GoLeMo, Lo24}]\label{diverg}
	Let $X$ be a separable infinite-dimensional complex Banach space, and let $T \in \mathcal{L}(X)$. Suppose that
	\[
	\sigma_{\ell e}(T) \cap \overline{\mathbb{D}} = \emptyset.
	\]
	Then every infinite-dimensional closed subspace $Z \subset X$ contains a vector $x \in Z$ such that
	\[
	\lim_{n \to \infty} \|T^n x\| = \infty.
	\]
\end{lemma}

\subsection{Spaceability of absolutely $(\mu_m)_m$-irregular vectors }

\begin{theorem} \label{condi-suf-space-irre}
	Let $X$ be a separable infinite-dimensional real or complex Banach space and let
	$T \in \mathcal{L}(X)$. Let $(\mu_m)_{m\in\mathbb{N}}$ satisfy condition \emph{(M1)}.
	Assume that $T$ is not absolutely $(\mu_m)_{m\in\mathbb{N}}$-bounded and that the
	following conditions hold:
	\begin{itemize}
		\item The set \(\displaystyle{\Bigl\{x \in X : \lim_{m\to\infty}\int_{\mathbb{N}} \|T^{t}x\|\,d\mu_m=0\Bigr\}}\)
		is dense in $X$.
		\item There exists a decreasing sequence $(E_n)_{n\in\mathbb{N}}$ of
		infinite-dimensional closed subspaces of $X$ such that
		\[
		M:=\sup_{n\in\mathbb{N}} \int_{\mathbb{N}} \bigl\|T^{t}|_{E_n}\bigr\|\,d\mu_n(t)
		< \infty.
		\]
	\end{itemize}
	Then there exist an infinite-dimensional closed subspace $F \subset X$ and a
	strictly increasing sequence of positive integers $(\theta_n)_{n\in\mathbb{N}}$
	such that
	\[
	\lim_{n\to\infty} \int_{\mathbb{N}} \|T^{t}x\|\,d\mu_{\theta_n}=0\;
	\text{and}\;
	\limsup_{m\to\infty} \int_{\mathbb{N}} \|T^{t}x\|\,d\mu_m =\infty\;\;
	\forall x\in F\setminus\{0\}.
	\]
\end{theorem}

\begin{proof}
	By assumption, \(T\) is not absolutely \((\mu_{m})\)-bounded. Hence, there exist a vector
	\(y\in X\) and a strictly increasing sequence of positive integers
	\((m_{k})_{k}\) such that
	\[
	\lim_{k\to\infty}\int_{\mathbb{N}} \|T^{t}y\|\, d\mu_{m_{k}}(t)=\infty.
	\]
	Thus, by Corollary~\ref{not-abs-bounded},
	\[
	G=\Bigl\{x\in X:\limsup_{k\to\infty}\int_{\mathbb{N}} \|T^{t}x\|\, d\mu_{m_{k}}(t)=\infty\Bigr\}
	\]
	is a residual subset of \(X\).
	
	Now set \(\mathcal{P}:=\mathbb{N}\) and
	\[
	\mathcal{S}:=\{(n_{k})_{k}:\ (n_{k})_{k}\ \text{is a subsequence of } (m_{k})_{k}\}.
	\]
	Define \(f,g:X\times \mathcal{P}\rightarrow [0,\infty)\) by
	\[
	f(x,m)=g(x,m):=\int_{\mathbb{N}} \|T^{t}x\|\, d\mu_{m}(t),
	\]
	and define \(J:\mathcal{P}\rightarrow \mathbb{N}\) by \(J(m):=m\).
	
	Notice that for each \(n\in \mathbb{N}\) and \(x\in E_{n}\),
	\[
	g(x,n)=\int_{\mathbb{N}} \|T^{t}x\|\, d\mu_{n}(t)
	\leq \|x\| \int_{\mathbb{N}} \|T^{t}|_{E_{n}}\|\, d\mu_{n}(t)
	\leq M\|x\|.
	\]
	It is straightforward to verify that condition~(\ref{condi J-g}) holds.
	
	Therefore, by Lemma~\ref{space-cri}, there exist an infinite-dimensional
	closed subspace \(F\subset X\) and a strictly increasing sequence of
	positive integers \((\theta_{n})_{n\in\mathbb{N}}\) such that
	\[
	\lim_{n\to\infty} \int_{\mathbb{N}} \|T^{t}x\|\, d\mu_{\theta_{n}}(t)=0
	\quad \forall x\in F,
	\]
	and
	\[
	\limsup_{m\to\infty} \int_{\mathbb{N}} \|T^{t}x\|\, d\mu_{m}(t)=\infty
	\quad \forall x\in F\setminus\{0\}.
	\]
	This completes the proof.
\end{proof}

\begin{theorem}\label{equi-space-irregu}
	Let $X$ be a separable infinite-dimensional complex Banach space and let
	$T \in \mathcal{L}(X)$. Assume that $T$ is not power bounded and that the set
	\[
	\{x \in X : T^{n}x \xrightarrow[n\to\infty]{} 0\}
	\]
	is dense in $X$. Then the following assertions are equivalent.
	
	\begin{enumerate}
		
		\item The set
		\[
		\Bigl\{x\in X : \liminf_{n\to\infty} \|T^{n}x\| = 0\Bigr\}
		\]
		is spaceable.
		
		\item The set of all irregular vector for \(T\) is spaceable.
		
		\item If $T$ is not absolutely $(\mu_m)$-bounded and $(\mu_m)$ satisfies \emph{(M1)}, then there exist an
		infinite-dimensional closed subspace $F\subset X$ and a strictly increasing
		sequence of positive integers $(\theta_n)_{n\in\mathbb N}$ such that
		\[
		\lim_{m\to\infty} \int_{\mathbb{N}} \|T^{t}x\|\, d\mu_{\theta_m}(t) = 0
		\qquad \forall x\in F,
		\]
		and
		\[
		\limsup_{m\to\infty} \int_{\mathbb{N}} \|T^{t}x\|\, d\mu_{m}(t) = \infty
		\qquad \forall x\in F\setminus\{0\}.
		\]
		
		\item There exist a strictly increasing sequence of positive integers
		$(\theta_n)_{n\in\mathbb N}$ and an infinite-dimensional closed
		subspace $E\subset X$ such that
		\[
		\sup_{n\in\mathbb N} \bigl\|T^{\theta_n}|_{E}\bigr\| < \infty.
		\]
		
		\item The left essential spectrum of $T$ intersects the closed unit disk
		$\overline{\mathbb{D}}$.
		
	\end{enumerate}
\end{theorem}

\begin{proof}
	The implications $(3)\Rightarrow(2)$ and $(2)\Rightarrow(1)$ are immediate.
	Now we show that $(3)\Rightarrow(4)$. By hypothesis, $T$ is not absolutely
	$(\delta_m)_{m\in\mathbb N}$-bounded. Then, there exist a strictly increasing sequence of positive integers
	$(\theta_m)_m$ and an infinite-dimensional closed subspace $F\subset X$
	such that
	\[
	\lim_{m\to\infty}\|T^{\theta_m}x\|=0
	\qquad \text{for each } x\in F.
	\]
	Therefore, by the Banach--Steinhaus Theorem, it follows that $(4)$ holds.
	
	The implications $(1)\Rightarrow(5)$ and $(4)\Rightarrow(5)$ are direct
	consequences of Lemma~\ref{diverg}.
	
	It therefore remains to prove that $(5)\Rightarrow(3)$. Assume that there
	exists $\lambda \in \sigma_{\ell e}(T) \cap \overline{\mathbb{D}}$;
	equivalently, $T-\lambda$ is not a left-Fredholm operator. By
	\cite[Proposition D.3.4]{BaMa}, there exist an infinite-dimensional closed
	subspace $E$ and a compact operator $K \in \mathcal{L}(X)$ such that
	\[
	(T-K)|_E = \lambda \mathrm{Id}|_E .
	\]
	In particular, $\|(T-K)^{n}|_E\|\leq 1$ for each \(n\in \mathbb{N}\).
	
	For each $n$ we can write
	\[
	T^{n} = (T-K)^{n} + K_{n},
	\]
	where $K_{n}$ is compact. Consider the sequence $\{K_n\}_{n \in \mathbb{N}}$
	of compact operators on $X$ arranged as
	\[
	K_{1}, K_{2}, \ldots, K_{n}, \cdots .
	\]
	According to \cite[Lemma 8.13]{BaMa}, there exists a non-increasing sequence
	$\{E_n\}_{n \in \mathbb{N}}$ of finite-codimensional closed subspaces of $E$
	such that $\|K_n|_{E_n}\| \leq 1$.
	
	Fix any $n$. Then
	\begin{align*}
		\|T^{n}|_{E_n}\|
		= \|(T-K)^{n}|_{E_n} + K_{n}|_{E_n}\| 
		\leq \|(T-K)^{n}|_{E_n}\| + \|K_{n}|_{E_n}\| 
		\leq 2 .
	\end{align*}
	Since $n$ is arbitrary, it follows that
	\[
	\sup_{n\geq 1} \|T^{n}|_{E_{n}}\| \leq 2 .
	\]
	
	According to Theorem~\ref{condi-suf-space-irre}, there exist an
	infinite-dimensional closed subspace $F\subset X$ and a strictly increasing
	sequence of positive integers $(\theta_n)_{n\in\mathbb N}$ such that
	\[
	\lim_{m\to\infty} \int_{\mathbb{N}} \|T^{t}x\|\, d\mu_{\theta_m}(t) = 0
	\; \text{and} \limsup_{m\to\infty} \int_{\mathbb{N}} \|T^{t}x\|\, d\mu_{m}(t) = \infty
	\;\; \forall x\in F\setminus\{0\}.
	\]
	This concludes the proof.
\end{proof}

\subsection{Spaceability of Distributionally \((\mu_m)\)-Irregular Vectors}

\begin{theorem}\label{condi-suf-spa-dis-irre}
	Let $X$ be a separable infinite-dimensional real or complex Banach space and let
	$T \in \mathcal{L}(X)$. Let $(\mu_{m})$ satisfy condition \emph{(M1)}. Assume that there exists a sequence of finite subsets of $\mathbb{N}$, $(D_{k})_{k}$, and a  sequence of integers $(m_{k})_{k}$ in \(\mathbb{N}\) with \(m_{k}\rightarrow \infty\) such that \(	\mu_{m_{k}}(D_{k})>1-\frac{1}{k}\) and such that the set
	\[
	\Bigl\{x \in X : \lim_{k\rightarrow \infty} \max_{t\in D_{k}} \Vert T^{t}x\Vert =0\Bigr\}
	\]
	is dense in $X$. Assume moreover that the following conditions hold:
	\begin{itemize}
		\item There exists \(y\in X\) such that
		\[
		\lim_{k\rightarrow \infty} \min_{t\in D_{k}} \Vert T^{t}y\Vert =\infty .
		\]
		
		\item There exists a decreasing sequence $(E_n)_{n\in\mathbb{N}}$ of
		infinite-dimensional closed subspaces of $X$ such that
		\[
		\sup_{n\in\mathbb{N}} \bigl\|T^{n}|_{E_n}\bigr\| < \infty .
		\]
	\end{itemize}
	
	Then there exist an infinite-dimensional closed subspace $F \subset X$
	and a set \(A\subset \mathbb{N}\) with \(\displaystyle \limsup_{m\to\infty}\mu_{m}(A)=1\) such that
	\[
	\lim_{n\in A} T^{n}x=0 \qquad \forall x\in F,
	\]
	and for every \(x\in F\setminus\{0\}\) there exists a set \(B\subset \mathbb{N}\) with
	\(\displaystyle \limsup_{m\to\infty}\mu_{m}(B)=1\) such that
	\[
	\lim_{n\in B} \Vert T^{n}x\Vert=\infty .
	\]
\end{theorem}

\begin{proof}
	Notice that the set
	\[
	G:=\{x\in X: \limsup_{k\to\infty} \min_{t\in D_k} \|T^{t}x\|=\infty\}
	\]
	is residual in \(X\). Let
	\[
	\mathcal{P}:=\{A\subset \mathbb{N}: A \text{ finite}\}
	\]
	and define
	\[
	\mathcal S :=
	\left\{
	(A_k)_k : (A_k)_k \text{ is a subsequence of } (D_k)_k
	\right\}.
	\]
	Define \(f,g : X\times \mathcal P \to [0,\infty)\) by
	\[
	f(x,A):=\min_{t\in A}\|T^{t}x\|,
	\qquad
	g(x,A):=\max_{t\in A}\|T^{t}x\|.
	\]
	Also define \(J: \mathcal{P}\to\mathbb{N}\) by
	\[
	J(A):=\max\{n: n\in A\}.
	\]
	Let \(n\in\mathbb{N}\), \(x\in E_n\), and suppose that \(J(A)=n\). Then
	\begin{align*}
		g(x,A)
		= \max_{t\in A}\|T^t x\| 
		& \le \max_{1\le t\le n}\|T^t x\| \\
		&	\le \max_{1\le t\le n}\|T^t|_{E_t}\|\,\|x\| \\
		&	\le \sup_{m\in\mathbb{N}}\|T^m|_{E_m}\|\,\|x\|.
	\end{align*}
	On the other hand, it is easy to verify that condition (\ref{condi J-g}) holds.
	
	Therefore, by Lemma \ref{space-cri}, there exist an infinite-dimensional closed subspace \(F\subset X\) and a set \(A\subset\mathbb N\) with
	\[
	\limsup_{m\to\infty}\mu_m(A)=1
	\]
	such that
	\[
	\lim_{n\in A}T^n x =0
	\qquad \text{for all } x\in F,
	\]
	and for every \(x\in F\setminus\{0\}\) there exists a set \(B\subset\mathbb N\) with
	\[
	\limsup_{m\to\infty}\mu_m(B)=1
	\]
	such that
	\[
	\lim_{n\in B}\|T^n x\|=\infty .
	\]
	This concludes the proof.
\end{proof}

\begin{theorem}
	Let $X$ be a separable infinite-dimensional complex Banach space and let
	$T \in \mathcal{L}(X)$. Assume  $T$ is not power bounded and that the set
	\[
	\Bigl\{x \in X : T^{n}x\xrightarrow[n\rightarrow \infty]{} 0\Bigr\}
	\]
	is dense in $X$. Then the following statements are equivalent:
	\begin{enumerate}
		\item The set
		\[
		\Bigl\{x\in X : \liminf_{n\to\infty} \|T^{n}x\| = 0\Bigr\}
		\]
		is spaceable.
		
		\item The set of all irregular vector for \(T\) is spaceable.
		
		\item If $T$ admits a vector with $(\mu_{m})$-distributionally unbounded orbit and \((\mu_{m})\) satisfies \emph{(M1)}, then there exist an infinite-dimensional closed subspace $F \subset X$
		and a set \(A\subset \mathbb{N}\) with \(\displaystyle{\limsup_{m\to\infty}\mu_{m}(A)=1}\) such that
		\[
		\lim_{n\in A} T^{n}x=0 \qquad \text{for all } x\in F,
		\]
		and for each \(x\in F\setminus\{0\}\) there exists a set \(B\subset \mathbb{N}\) with \(\displaystyle{\limsup_{m\rightarrow \infty}\mu_{m}(B)}=1\) such that
		\[
		\lim_{n\in B} \Vert T^{n}x\Vert=\infty .
		\]
		
		\item There exists a decreasing sequence $(E_n)_{n\in\mathbb{N}}$ of
		infinite-dimensional closed subspaces of $X$ such that
		\[
		\sup_{n\in\mathbb{N}} \bigl\|T^{n}|_{E_n}\bigr\| < \infty .
		\]
		\item The left essential spectrum of $T$ intersects the closed unit disk
		$\overline{\mathbb{D}}$.
	\end{enumerate}
\end{theorem}

\begin{proof}
	The implications (3) $\Rightarrow$ (2) and (2) $\Rightarrow$ (1) are immediate. The implication (3) $\Rightarrow$ (4) follows from the Banach–Steinhaus theorem. On the other hand, according to Lemma \ref{diverg}, (1) implies (5) and (4) implies (5).
	
	We now show that (5) implies (3). Proceeding as in the proof of the theorem \ref{equi-space-irregu}, there exists a decreasing sequence $(E_n)_{n\in\mathbb{N}}$ of infinite-dimensional closed subspaces of $X$ such that
	\[
	\sup_{n\in\mathbb{N}} \bigl\|T^{n}|_{E_n}\bigr\|\leq 2 .
	\]
	
	According to Theorem \ref{condi-suf-spa-dis-irre}, there exist an infinite-dimensional closed subspace $F \subset X$
	and a set \(A\subset \mathbb{N}\) with \(\displaystyle{\limsup_{m\to\infty}\mu_{m}(A)=1}\) such that
	\[
	\lim_{n\in A} T^{n}x=0 \qquad \text{for all } x\in F,
	\]
	and for each \(x\in F\setminus\{0\}\) there exists a set \(B\subset \mathbb{N}\) with \(\displaystyle{\limsup_{m\rightarrow \infty}\mu_{m}(B)}=1\) such that
	\[
	\lim_{n\in B} \|T^{n}x\|=\infty .
	\]
	This concludes the proof.
\end{proof}



\section{Baire Theorem and Observation Schemes}\label{b.t.o.s}

In this section, \(I:=\mathbb{N}\), \(\Lambda:=\mathbb{N}\) and \(T_{t}:=T^{t}\) with \(t\in \mathbb{N}\) for some continuous linear operator on \(X\). In this way, without loss of generality, we assume throughout that
\[
\operatorname{supp}(\mu_m)\subset\{1,\dots,m\}
\qquad \text{for every } m\in \mathbb{N}.
\]

For each $m$, let $\mathcal{P}(\{1,\dots,m\})$ denote the simplex of
probability measures on $\{1,\dots,m\}$, equipped with the total
variation norm
\[
\|\mu-\nu\|_{\mathrm{TV}}
:=\frac12\sum_{\ell=1}^m
\bigl|\mu(\{\ell\})-\nu(\{\ell\})\bigr|.
\]

The natural ambient space for the lenses is the product
\[
\mathcal{M}:=\prod_{m=1}^\infty \mathcal{P}(\{1,\dots,m\}),
\]
endowed with the supremum metric
\[
\|(\mu_m)_m-(\nu_m)_m\|_\infty:=\sup_{m\ge1}\|\mu_m-\nu_m\|_{\mathrm{TV}}.
\]

With this metric, $\mathcal{M}$ is a complete metric space.

Our principal object of interest is the closed subset
\begin{equation}\label{eq:Upsilon}
	\Upsilon:=\Bigl\{(\mu_m)_m\in\mathcal{M}:\ 
	\lim_{m\to\infty}\mu_m(A)=0 \text{ for every finite } A\subset\mathbb{N}\Bigr\}.
\end{equation}

Consequently, $(\Upsilon,\|\cdot\|_\infty)$ is itself complete.

The first result of this section shows that dense \((\mu_{m})\)-distributional chaos exhibits a rigid dichotomy at the level of observation schemes: it is either present for all schemes or confined to a subset of first category in \(\Upsilon\).

\begin{theorem}\label{dicho-dist}
	Let \(X\) be a separable Fréchet space and let \(T \in \mathcal{L}(X)\). Then the set
	\[
	\left\{(\mu_{m})\in \Upsilon:\; T \text{ is densely }(\mu_{m})\text{-distributionally chaotic}\right\}
	\]
	is either equal to \(\Upsilon\) or a set of first category in \(\Upsilon\). Moreover, if \(X\) is a Banach space, then the latter alternative is closed and nowhere dense in \(\Upsilon\).
\end{theorem}

To prove the theorem, we require the following two propositions.

\begin{proposition}\label{closed-J}
	Let \(X\) be a Fréchet space, \(T \in \mathcal{L}(X)\), and \(\beta \in \mathbb{N}\). Then the set
	\[
	J_{\beta} := \{ (\mu_{m})_{m} \in \Upsilon : T \text{ admits a } (\mu_{m})\text{-distributionally } \beta\text{-unbounded vector} \}
	\]
	is either \(\Upsilon\) or closed and nowhere dense in \(\Upsilon\).
\end{proposition}

\begin{proof}
	We distinguish three cases. The first two correspond to \(J_{\beta}=\Upsilon\) or \(J_{\beta}=\emptyset\). It remains to analyze the case where \(\emptyset \neq J_{\beta} \subsetneq \Upsilon\).
	
	We claim that \(J_{\beta}\) is closed and nowhere dense in \(\Upsilon\). First, we show that \(\Upsilon \setminus J_{\beta}\) is open in \(\Upsilon\). Let \((\mu_{m})_{m} \in \Upsilon \setminus J_{\beta}\). Adopting the notation from the proof of Proposition \ref{dist-G-delta}, we have
	\[
	\bigcap_{k \in \mathbb{N}} R_{k,\beta}((\mu_{m})_{m}) = \emptyset.
	\]
	Thus,
	\[
	\bigcup_{k \in \mathbb{N}} \left( X \setminus R_{k,\beta}((\mu_{m})_{m}) \right) = X.
	\]
	By the Baire Category Theorem, there exist \(k_{0} \in \mathbb{N}\) and a non-empty open subset \(U \subset X\) such that
	\begin{align*}
		U &\subset X \setminus R_{k_{0}, \beta}((\mu_{m})_{m}) \\
		&= \left\{ x \in X : \mu_{m}(\{ t \in \mathbb{N} : \| T^{t}x \|_{\beta} > k_{0} \}) \leq 1 - \frac{1}{k_{0}}, \forall m \geq k_{0} \right\}.
	\end{align*}
	Let \((\nu_{m})_{m} \in \Upsilon\) with \(\| (\nu_{m})_{m} - (\mu_{m})_{m} \|_{\infty} < \frac{1}{4k_{0}}\). Observe that for \(x \in U\) and \(m \geq k_{0}\),
	\begin{align*}
		\nu_{m}(\{ t \in \mathbb{N} : \| T^{t}x \|_{\beta} > k_{0} \}) &\leq \frac{1}{2k_{0}} + \mu_{m}(\{ t \in \mathbb{N} : \| T^{t}x \|_{\beta} > k_{0} \}) \\
		&\leq 1 - \frac{1}{2k_{0}}.
	\end{align*}
	Since \(\{ t \in \mathbb{N} : \| T^{t}x \|_{\beta} > 2k_{0} \} \subset \{ t \in \mathbb{N} : \| T^{t}x \|_{\beta} > k_{0} \}\), it follows that
	\[
	\nu_{m}(\{ t \in \mathbb{N} : \| T^{t}x \|_{\beta} > 2k_{0} \}) \leq  1 - \frac{1}{2k_{0}}, \quad \forall x \in U, \forall m \geq 2k_{0}.
	\]
	Therefore, \(U \subset X \setminus R_{2k_{0}, \beta}((\nu_{m})_{m})\). If \((\nu_{m})_{m} \in J_{\beta}\), then by Proposition \ref{dist-unbounded} the set
	\[
	\bigcap_{k \in \mathbb{N}} R_{k,\beta}((\nu_{m})_{m})
	\]
	is residual. In particular, \(U \cap R_{2k_{0}, \beta}((\nu_{m})_{m}) \neq \emptyset\), which yields a contradiction. Hence, \((\nu_{m})_{m} \notin J_{\beta}\). This shows that \(\Upsilon \setminus J_{\beta}\) is open in \(\Upsilon\).
	
	It remains to show that \(\Upsilon \setminus J_{\beta}\) is dense in \(\Upsilon\). Fix \((\eta_{m})_{m} \in J_{\beta}\) and choose \((\nu_{m})_{m} \in \Upsilon \setminus J_{\beta}\). For each \(0 < s < 1\), one readily verifies that \((s\nu_{m} + (1-s)\eta_{m})_{m} \in \Upsilon \setminus J_{\beta}\) and
	\[
	\| (s\nu_{m} + (1-s)\eta_{m})_{m} - (\eta_{m})_{m} \|_{\infty} \leq s.
	\]
	Letting \(s \to 0\), the density of \(\Upsilon \setminus J_{\beta}\) follows. Therefore, \(J_{\beta}\) is closed and nowhere dense in \(\Upsilon\). This concludes the proof.
\end{proof}

\begin{proposition}\label{closed-N}
	Let \(X\) be a Fréchet space and let \(T \in \mathcal{L}(X)\). Define
	\[
	\mathcal{N}:=\left\{(\mu_{m})\in \Upsilon:\; T \text{ is dense } (\mu_{m})\text{-distributionally near to zero } \right\}.
	\]
	Then \(\mathcal{N}\) is either equal to \(\Upsilon\) or it is closed and nowhere dense in \(\Upsilon\).
\end{proposition}

\begin{proof}
	The argument follows the same scheme as in the proof of Proposition~\ref{closed-J}. We distinguish three cases: either \(\mathcal{N}=\Upsilon\), or \(\mathcal{N}=\emptyset\), or \(\emptyset \neq \mathcal{N} \subsetneq \Upsilon\). We focus on the latter.
	
	We claim that \(\mathcal{N}\) is closed and nowhere dense in \(\Upsilon\). To this end, it suffices to show that \(\Upsilon\setminus \mathcal{N}\) is open and dense.
	
	Let \((\mu_{m})\in \Upsilon\setminus \mathcal{N}\). Then
	\[
	\bigcap_{k \in \mathbb{N}} S_{k}((\mu_{m}))
	\]
	is a \(G_{\delta}\)-subset of \(X\) which is not dense. Hence, by the Baire category theorem, there exist \(k_{0}\in \mathbb{N}\) and a non-empty open set \(U\subset X\) such that
	\begin{align*}
		U &\subset X\setminus S_{k_{0}}((\mu_{m}))\\
		&= \left\{ x\in X : 
		\mu_{m}\big(\{t\in \mathbb{N} : \rho(T^{t}x,0)<\tfrac{1}{k_{0}}\}\big)\leq 1-\tfrac{1}{k_{0}}, \ \forall m\geq k_{0} \right\}.
	\end{align*}
	Notice that for every \((\nu_{m})\in B_{\infty}((\mu_{m}),\frac{1}{4k_{0}})\),
	\[
	U\subset X\setminus S_{2k_{0}}((\nu_{m})).
	\]
	In particular,
	\[
	B_{\infty}((\mu_{m}),\frac{1}{4k_{0}})\subset \Upsilon\setminus \mathcal{N},
	\]
	which shows that \(\Upsilon\setminus \mathcal{N}\) is open.
	
	To prove density, let \((\eta_{m})\in \mathcal{N}\) and choose \((\nu_{m})\in \Upsilon\setminus \mathcal{N}\). For \(0<s<1\), define
	\[
	\mu^{(s)}_{m}:=s\nu_{m}+(1-s)\eta_{m}.
	\]
	Then \((\mu^{(s)}_{m})\in \Upsilon\setminus \mathcal{N}\) and \((\mu^{(s)}_{m})\to (\eta_{m})\) as \(s\to 0^{+}\). Hence, \(\Upsilon\setminus \mathcal{N}\) is dense.
	
	Therefore, \(\mathcal{N}\) is closed and nowhere dense in \(\Upsilon\), which completes the proof.
\end{proof}

\begin{proof}[Proof of Theorem \ref{dicho-dist}]
	We begin by observing that the following assertions are equivalent by Theorem \ref{dense-mu_i-distri}:
	\begin{itemize}
		\item \(T\) is densely \((\mu_{m})\)-distributionally chaotic;
		\item \((\mu_{m})\in \mathcal{N}\cap\Bigl(\bigcup_{\beta\in\mathbb{N}} J_{\beta}\Bigr)\).
	\end{itemize}
	
	By Propositions~\ref{closed-N} and \ref{closed-J}, the sets \(\mathcal{N}\) and \(J_{\beta}\) satisfy a dichotomy: each of them is either equal to \(\Upsilon\) or meagre in \(\Upsilon\); moreover, in the Banach setting, they are closed and nowhere dense whenever they are not all of \(\Upsilon\).
	
	It follows that \(\mathcal{N}\cap\bigl(\bigcup_{\beta} J_{\beta}\bigr)\) is either equal to \(\Upsilon\) or a set of first category in \(\Upsilon\). In the Banach case, this set is in addition closed and nowhere dense.
	
	This completes the proof.
\end{proof}

\begin{theorem}\label{G_delta-dense-LY}
	Let \(X\) be a separable Banach space and let \(T \in \mathcal{L}(X)\). Then the set
	\[
	\left\{(\mu_{m})\in \Upsilon:\; T \text{ is densely } (\mu_{m})\text{--Li--Yorke chaotic}\right\}
	\]
	is a \(G_{\delta}\)-subset of \(\Upsilon\).
\end{theorem}

The proof is based on the following two propositions.

\begin{proposition}\label{abs-near to zero}
	Let \(X\) be a separable Fréchet space and let \(T\in \mathcal{L}(X)\). Then the set
	\[
	\mathcal{A}=\left\{(\mu_{m})\in \Upsilon:\;
	\left\{x\in X:\; \liminf_{m\to\infty}\int\rho(T^{t}x,0)\,d\mu_{m}(t)=0\right\}
	\text{ is dense in } X\right\}
	\]
	is a \(G_{\delta}\)-subset of \(\Upsilon\). Moreover, either \(\mathcal{A}=\Upsilon\) or \(\Upsilon\setminus \mathcal{A}\) is dense in \(\Upsilon\).
\end{proposition}

\begin{proof}
	We begin by observing that, for a given \((\mu_m)_m \in \Upsilon\), the set
	\[
	\left\{x\in X:\; \liminf_{m\to\infty}\int \rho(T^{t}x,0)\,d\mu_{m}(t)=0\right\}
	\]
	can be written as
	\[
	\bigcap_{k\in \mathbb{N}} B_k((\mu_m)_m),
	\; \text{where} \;
	B_k((\mu_m)_m)
	=
	\bigcup_{m\ge k}
	\left\{x\in X: \int \rho(T^{t}x,0)\,d\mu_{m}(t)<\frac{1}{k}\right\}.
	\]
	Each set \(B_k((\mu_m)_m)\) is open in \(X\), and
	\[
	X\setminus B_k((\mu_m)_m)
	=
	\left\{x\in X: \int \rho(T^{t}x,0)\,d\mu_{m}(t)\ge \frac{1}{k},\ \forall m\ge k\right\}.
	\]
	
	Let \(\{U_\ell\}_{\ell\in\mathbb{N}}\) be a countable basis of nonempty open subsets of \(X\). For \(\ell,k\in\mathbb{N}\), define
	\[
	M_{\ell,k}
	:=
	\left\{(\mu_m)_m \in \Upsilon:\; U_\ell \subset X\setminus B_k((\mu_m)_m)\right\}.
	\]
	
	\medskip
	\noindent
	\textbf{Claim 1.} Each \(M_{\ell,k}\) is closed in \(\Upsilon\).
	
	\medskip
	Let \((\mu_{m,n})_m \to (\mu_m)_m\) in \(\Upsilon\), with \((\mu_{m,n})_m \in M_{\ell,k}\).  
	Fix \(x\in U_\ell\) and \(m\ge k\). Then
	\[
	\int \rho(T^{t}x,0)\,d\mu_{m,n}(t)
	\longrightarrow
	\int \rho(T^{t}x,0)\,d\mu_m(t),
	\]
	and since each \((\mu_{m,n})_m \in M_{\ell,k}\),
	\[
	\int \rho(T^{t}x,0)\,d\mu_{m,n}(t)\ge \frac{1}{k}.
	\]
	Passing to the limit yields
	\[
	\int \rho(T^{t}x,0)\,d\mu_m(t)\ge \frac{1}{k}.
	\]
	Hence \((\mu_m)_m \in M_{\ell,k}\), and \(M_{\ell,k}\) is closed.
	
	\medskip
	\noindent
	\textbf{Claim 2.} 
	\[
	\Upsilon\setminus \mathcal{A}
	=
	\bigcup_{\ell,k\in\mathbb{N}} M_{\ell,k}.
	\]
	
	\medskip
	If \((\mu_m)_m \notin \mathcal{A}\), then
	\[
	\bigcap_{k\in\mathbb{N}} B_k((\mu_m)_m)
	\]
	is not dense in \(X\). By the Baire category theorem, there exist a nonempty open set \(W\subset X\) and \(k\in\mathbb{N}\) such that
	\[
	W \subset X\setminus B_k((\mu_m)_m).
	\]
	Choosing \(U_\ell \subset W\), we obtain \((\mu_m)_m \in M_{\ell,k}\).  
	The converse inclusion is immediate.
	
	\medskip
	It follows that \(\Upsilon\setminus \mathcal{A}\) is an \(F_\sigma\)-set, and therefore \(\mathcal{A}\) is a \(G_\delta\)-subset of \(\Upsilon\).
	
	\medskip
	Finally, assume that \(\mathcal{A}\neq \Upsilon\). Let \((\mu_m)_m \in \mathcal{A}\) and \((\nu_m)_m \in \Upsilon\setminus \mathcal{A}\). Then there exist \(\ell,k\in\mathbb{N}\) such that
	\[
	U_\ell \subset
	\left\{x\in X:\; \int \rho(T^{t}x,0)\,d\nu_m(t)\ge \frac{1}{k},\ \forall m\ge k\right\}.
	\]
	
	For \(0<s<1\), define
	\[
	(\eta_m^{(s)})_m := \bigl(s\nu_m + (1-s)\mu_m\bigr)_m.
	\]
	Then \((\eta_m^{(s)})_m \to (\mu_m)_m\) as \(s\to 0^+\). Moreover, for \(m\ge k\) and \(x\in U_\ell\),
	\[
	\int \rho(T^{t}x,0)\,d\eta_m^{(s)}(t)
	=
	s \int \rho(T^{t}x,0)\,d\nu_m(t)
	+
	(1-s)\int \rho(T^{t}x,0)\,d\mu_m(t)
	\ge \frac{s}{k}.
	\]
	Thus, choosing \(k_1 > k/s\), we obtain
	\[
	U_\ell \subset X\setminus B_{k_1}((\eta_m^{(s)})_m),
	\]
	so that \((\eta_m^{(s)})_m \in \Upsilon\setminus \mathcal{A}\).
	
	This shows that \(\Upsilon\setminus \mathcal{A}\) is dense in \(\Upsilon\), completing the proof.
\end{proof}

\begin{proposition}\label{set-abs-bounded}
	Let \(X\) be a Fréchet space and \(T\in \mathcal{L}(X)\). Then the set
	\[
	\left\{(\mu_{m})_{m}\in \Upsilon:\; T \text{ is not absolutely }(\mu_{m})\text{-bounded}\right\}
	\]
	is either empty or residual in \(\Upsilon\). Moreover, if \(X\) is a Banach space, then the latter alternative is a dense \(G_{\delta}\)-subset of \(\Upsilon\).
\end{proposition}

\begin{proof}
	We distinguish two cases.
	
	If \(T\) is power bounded, then it is absolutely \((\mu_{m})\)-bounded for every \((\mu_{m})\in \Upsilon\), and hence the set under consideration is empty.
	
	Assume now that \(T\) is not power bounded. Then there exists \(\beta_{0}\in \mathbb{N}\) such that \(T\) is not \(\beta_{0}\)-absolutely \((\delta_{m})\)-bounded.
	
	For each \(k,\ell \in \mathbb{N}\), define
	\[
	B_{k,\ell}
	:=
	\left\{
	(\mu_m)_m \in \Upsilon :
	\sup_{\substack{x\in X \\ \rho(x,0)<1/k}}
	\ \sup_{m \in \mathbb{N}}
	\int \|T^{t}x\|_{\ell} \, d\mu_m(t)
	\leq 1
	\right\}.
	\]
	Each \(B_{k,\ell}\) is closed in \(\Upsilon\), and
	\[
	T \text{ is absolutely } (\mu_m)\text{-bounded}
	\quad \Longleftrightarrow \quad
	(\mu_m)_m \in \bigcap_{\ell\in \mathbb{N}} \bigcup_{k\in \mathbb{N}} B_{k,\ell}.
	\]
	
	Fix \(\beta_{0}\) as above and set
	\[
	\mathfrak{U}(\beta_{0})
	:=
	\bigcap_{k\in \mathbb{N}} \bigl(\Upsilon \setminus B_{k,\beta_{0}}\bigr).
	\]
	Then \(\mathfrak{U}(\beta_{0})\) is a \(G_\delta\)-subset of \(\Upsilon\).
	
	We claim that \(\mathfrak{U}(\beta_{0})\) is dense in \(\Upsilon\). Let \((\nu_m)_m \in \Upsilon\) and \(k\in\mathbb{N}\). For \(0<s\le 1\), define
	\[
	(\mu_m^{(s)})_m := \bigl(s\delta_m + (1-s)\nu_m\bigr)_m.
	\]
	Since \(T\) is not \(\beta_{0}\)-absolutely \((\delta_m)\)-bounded, it follows that \((\mu_m^{(s)})_m \notin B_{k,\beta_{0}}\) for every \(s>0\). Hence \((\mu_m^{(s)})_m \in \mathfrak{U}(\beta_{0})\).
	
	Moreover,
	\[
	\| (\mu_m^{(s)})_m - (\nu_m)_m \|_\infty
	\le s \sup_{m} \|\delta_m - \nu_m\|_{\mathrm{TV}}
	\longrightarrow 0 \quad \text{as } s\to 0^{+},
	\]
	which shows that \((\nu_m)_m \in \overline{\mathfrak{U}(\beta_{0})}\). Therefore, \(\mathfrak{U}(\beta_{0})\) is dense in \(\Upsilon\).
	
	Finally, by construction, every \((\mu_m)_m \in \mathfrak{U}(\beta_{0})\) satisfies that \(T\) is not absolutely \((\mu_m)\)-bounded. Hence the set under consideration contains the residual set \(\mathfrak{U}(\beta_{0})\), and the proof is complete.
\end{proof}

\begin{proof}[Proof of Theorem \ref{G_delta-dense-LY}]
	By Theorem~\ref{equiv-dense}, \(T\) is densely \((\mu_{m})\)--Li--Yorke chaotic if and only if \((\mu_{m})\in \mathcal{A}\) and \(T\) is not absolutely \((\mu_{m})\)-bounded. In other words, the set under consideration is
	\[
	\mathcal{A}\cap \left\{(\mu_{m})_{m}\in \Upsilon:\; T \text{ is not absolutely }(\mu_{m})\text{-bounded}\right\}.
	\]
	By Propositions~\ref{abs-near to zero} and~\ref{set-abs-bounded}, this set is a \(G_{\delta}\)-subset of \(\Upsilon\).
\end{proof}

\begin{theorem}\label{gip}
	Let \(X\) be a separable infinite-dimensional real or complex Fréchet space and let \(T \in \mathcal{L}(X)\). Assume that the set \[ \{ x \in X : T^{n}x \xrightarrow[n\to\infty]{} 0 \} \] is dense in \(X\). Then exactly one of the following alternatives holds: 
	\begin{enumerate} 
		\item \(T^{n}x \xrightarrow[n\to\infty]{} 0\) for every \(x \in X\); 
		\item the set \[ \bigl\{ (\mu_m)_m \in \Upsilon : T \text{ admits a dense absolutely } (\mu_m)_m\text{-irregular manifold} \bigr\} \] is residual in \(\Upsilon\). \end{enumerate}
\end{theorem}

\begin{proof}
	If \(T\) is power bounded, then, since the set
	\[
	\{ x \in X : T^{n}x \to 0 \}
	\]
	is dense in \(X\), it follows that \(T^{n}x \to 0\) for every \(x \in X\).
	
	Assume now that \(T\) is not power bounded. By Proposition~\ref{set-abs-bounded}, there exists \(\beta_{0}\in \mathbb{N}\) such that the set \(	\mathfrak{U}(\beta_{0})\)
	is a dense \(G_{\delta}\)-subset of \(\Upsilon\), and for every \((\nu_m)_m \in \mathfrak{U}(\beta_{0})\), the operator \(T\) is not absolutely \((\nu_m)\)-bounded.
	
	By Theorem~\ref{dense-lineable}, it follows that for every \((\nu_m)_m \in \mathfrak{U}(\beta_{0})\), the operator \(T\) admits a dense absolutely \((\nu_m)\)-irregular manifold. Hence,
	\[
	\mathfrak{U}(\beta_{0}) \subset 
	\bigl\{ (\mu_m)_m \in \Upsilon : T \text{ admits a dense absolutely } (\mu_m)\text{-irregular manifold} \bigr\},
	\]
	which shows that the latter set is residual in \(\Upsilon\).
\end{proof}

\begin{theorem}[Trichotomy for Observation Schemes]\label{TOS}
	Let \(X\) be a separable real or complex Fréchet space and let \(T \in \mathcal{L}(X)\). Assume that the set
	\[
	\{x \in X : T^{n}x \to 0\}
	\]
	is dense in \(X\). Then exactly one of the following alternatives holds:
	\begin{enumerate}
		\item[(I)] \(T^{n}x \to 0\) for every \(x \in X\);
		\item[(II)] For every sequence \((\mu_m)_m \in \Upsilon\), \(T\) admits both a dense absolutely \((\mu_m)\)-irregular manifold and a dense \((\mu_m)\)-distributionally irregular manifold;
		\item[(III)] There exists a residual subset \(\Upsilon_0 \subset \Upsilon\) such that for every \((\mu_m)_m \in \Upsilon_0\), \(T\) admits a dense absolutely \((\mu_m)\)-irregular manifold while admitting no \(\beta\)-distributionally \((\mu_{m})\)-irregular vector for any \(\beta \in \mathbb{N}\).
	\end{enumerate}
\end{theorem}

\begin{proof}
	If \(T\) is power bounded, then \(T^{n}x \to 0\) for every \(x \in X\). Otherwise, there exists \(\beta_{0} \in \mathbb{N}\) such that \(T\) is not \(\beta_{0}\)-absolutely \((\delta_{m})\)-bounded. By Proposition \ref{closed-J}, we distinguish two cases.
	
	\textbf{Case 1.} There exists \(\beta \in \mathbb{N}\) such that \(J_{\beta} = \Upsilon\). Let \((\mu_{m})_m \in \Upsilon\). Then \(T\) admits a \((\mu_{m})\)-distributionally irregular vector with respect to \(\|\cdot\|_{\beta}\). By Proposition \ref{beta-desi}, \(T\) is not \(\beta\)-absolutely \((\mu_{m})\)-bounded. Therefore, by Theorem \ref{dense-lineable} and Theorem \ref{dense-disti}, \(T\) admits both a dense absolutely \((\mu_{m})\)-irregular manifold and a dense \((\mu_{m})\)-distributionally irregular manifold.
	
	\textbf{Case 2.} \(J_{\beta}\) is closed and nowhere dense in \(\Upsilon\) for every \(\beta \in \mathbb{N}\). Then
	\[
	\mathcal{J} := \bigcup_{\beta \in \mathbb{N}} J_{\beta}
	\]
	is a meager subset of \(\Upsilon\). Consider the residual set \(\mathfrak{U}(T, \beta_{0})\) defined in the proof of Theorem \ref{gip}. Let
	\[
	\Upsilon_{0} := \mathfrak{U}(T, \beta_{0}) \setminus \mathcal{J}.
	\]
	Then \(\Upsilon_{0}\) is residual in \(\Upsilon\). By Theorem \ref{dense-lineable}, for every \((\mu_{m})_m \in \Upsilon_{0}\), \(T\) admits a dense absolutely \((\mu_{m})\)-irregular manifold. On the other hand, for each \(\beta \in \mathbb{N}\), \(T\) admits no \(\beta\)-distributionally \((\mu_{m})\)-irregular vector. This concludes the proof.
\end{proof}

\begin{lemma}\label{distrib-irr-manifolds}
	Let \(X\) be a separable Fréchet space and \(T \in \mathcal{L}(X)\). Assume that \(\{x \in X : T^{n}x \to 0\}\) is dense in \(X\), and there exists \(x \in X\) and \(\beta \in \mathbb{N}\) such that 
	\[
	\lim_{n \to \infty} \|T^{n}x\|_{\beta} = \infty.
	\]
	Then, for every \((\mu_m)_m \in \Upsilon\), \(T\) admits both a dense absolutely \((\mu_m)\)-irregular manifold and a dense \((\mu_m)\)-distributionally irregular manifold.
\end{lemma}	

\begin{proof}
	For each \((\mu_{m})_{m}\in \Upsilon\). It is clear that $T$ is $(\mu_{m})$-distributionally irregular and it is not absolutely $(\mu_{m})$-bounded. By Theorem~\ref{TOS}, the required conclusion follows.
\end{proof}	

\begin{theorem}
	Let \(X\) be a separable space and let \(T \in \mathcal{L}(X)\). Assume that the set
	\[
	\{x \in X : T^{n}x \to 0\}
	\]
	is dense in \(X\). Suppose that one of the following conditions holds:
	\begin{enumerate}
		\item[(a)] \(X\) is a Fréchet space and \(T\) admits an eigenvalue \(\lambda\) with \(|\lambda|>1\);
		\item[(b)] \(X\) is a Banach space and \(\sum_{n=1}^\infty \frac{1}{\|T^{n}\|} < \infty\), or \(X\) is a Hilbert space and \(\sum_{n=1}^\infty \frac{1}{\|T^{n}\|^{2}} < \infty\).
	\end{enumerate}
	Then, for every \((\mu_m)_m \in \Upsilon\), the operator \(T\) admits both a dense absolutely \((\mu_m)\)-irregular manifold and a dense \((\mu_m)\)-distributionally irregular manifold.
\end{theorem}

\begin{proof}
	It suffices to observe that, if condition (a) holds, then the hypotheses of Lemma~\ref{distrib-irr-manifolds} are satisfied. If condition (b) holds, then, by \cite{muller2009orbits}, the hypotheses of Lemma~\ref{distrib-irr-manifolds} are also satisfied.
\end{proof}

A relevant class of operators satisfying the density hypothesis together with condition (a) of the previous theorem is given by those fulfilling the Godefroy--Shapiro Criterion. We say that an operator \(T \in \mathcal{L}(X)\), acting on a separable space \(X\), satisfies the Godefroy--Shapiro Criterion if both subspaces
\[
\mathrm{span}\Big(\bigcup_{|\lambda|>1} \ker(T-\lambda)\Big)
\quad \text{and} \quad
\mathrm{span}\Big(\bigcup_{|\lambda|<1} \ker(T-\lambda)\Big)
\]
are dense in \(X\). It is shown in \cite{BaMa, GrPe} that every such operator is hypercyclic.

\begin{corollary}
	If \(T\) satisfies the Godefroy--Shapiro Criterion, then for every \((\mu_m)_m \in \Upsilon\), the operator \(T\) admits both a dense absolutely \((\mu_m)\)-irregular manifold and a dense \((\mu_m)\)-distributionally irregular manifold.
\end{corollary}

\begin{theorem}
	Let \(X\) be a separable Fr\'echet space and \(T \in \mathcal{L}(X)\). Assume that:
	\begin{enumerate}
		\item[(a)] The set \(\{x \in X : T^{n}x \to 0\}\) is dense in \(X\).
		\item[(b)] There exists a subset \(Y \subset X\), a map \(S \colon Y \to Y\) satisfying \(TSy = y\) for all \(y \in Y\), and a vector \(z \in Y \setminus \{0\}\) such that the series \(\sum_{n=1}^\infty T^n z\) and \(\sum_{n=1}^\infty S^n z\) converge unconditionally in \(X\).
	\end{enumerate}
	Then, for every \((\mu_m)_m \in \Upsilon\), \(T\) admits both a dense absolutely \((\mu_m)\)-irregular manifold and a dense \((\mu_m)\)-distributionally irregular manifold.
\end{theorem}

\begin{proof}
	We follow the approach used in the proof of \cite[Theorem~19]{BeBoMuPe13}. The authors show that for each sufficiently large \(k_{0} \in \mathbb{N}\), there exists a nonzero vector \(w_{k_{0}}\) such that \(T^{k_{0}}w_{k_{0}} = w_{k_{0}}\), and furthermore, there exists a sequence \((x_{k})\) converging to \(0\) such that for each \(0 \leq \ell < k_{0}\),
	\[
	T^{\ell}w_{k_{0}} = \lim_{j \to \infty} T^{\ell + k_{0}j}x_{k}.
	\]
	Choose \(\beta \in \mathbb{N}\) such that
	\[
	\varepsilon := \frac{1}{2} \min \left\{ \|T^{\ell}w_{k_{0}}\|_{\beta} : 0 \leq \ell < k_{0} \right\} > 0.
	\]
	Fix any \((\mu_{m})_{m} \in \Upsilon\). Since the orbit of each \(x_{k}\) accumulates at \(\{T^{\ell}w_{k_{0}} : 0 \leq \ell < k_{0}\}\), there exists a strictly increasing sequence of positive integers \((m_{k})_{k}\) such that
	\[
	\lim_{k \to \infty} \mu_{m_{k}} \left( \{t \in \mathbb{N} : \|T^{t}x_{k}\|_{\beta} > \varepsilon\} \right) = 1.
	\]
	By Proposition \ref{dist-unbounded}, there exists a \((\mu_m)\)-distributionally \(\beta\)-unbounded vector for \(T\). Hence, by Theorem \ref{TOS}, the required conclusion follows.
\end{proof}

Let \(X\) be a separable Fréchet space and let \(T \in \mathcal{L}(X)\). We say that \(T\) satisfies the \emph{Frequently Hypercyclicity Criterion} if there exists a dense set \(\mathcal{D} \subset X\) and a map \(S \colon \mathcal{D} \to \mathcal{D}\) such that
\begin{enumerate}
	\item \(\sum_{n=0}^\infty T^{n}x\) and \(\sum_{n=0}^\infty S^{n}x\) are unconditionally convergent for each \(x \in \mathcal{D}\);
	\item \(TS = I\) on \(\mathcal{D}\).
\end{enumerate}

\begin{remark}
	We note that if an operator satisfies the Frequently Hypercyclicity Criterion, then it is frequently hypercyclic, mixing, and Devaney chaotic \cite{bonilla2007frequently}.
\end{remark}

\begin{corollary}\label{fhc-irre}
	Let \(T \in \mathcal{L}(X)\) satisfy the Frequently Hypercyclicity Criterion. Then, for every \((\mu_m)_m \in \Upsilon\), \(T\) admits both a dense absolutely \((\mu_m)\)-irregular manifold and a dense \((\mu_m)\)-distributionally irregular manifold.
\end{corollary}

\begin{example}
	Consider \(\mathrm{H}(\mathbb{C})\), the space of entire functions endowed with the topology of uniform convergence on compact sets. It is well known that \(\mathrm{H}(\mathbb{C})\) is a Fréchet space whose canonical defining family of seminorms \((\|\cdot\|_{\ell})_{\ell\in\mathbb{N}}\) is given by
	\[
	\|f\|_{\ell} := \sup_{|z| \leq \ell} |f(z)|, \qquad \ell \in \mathbb{N}.
	\]
	Let \(T \colon \mathrm{H}(\mathbb{C}) \to \mathrm{H}(\mathbb{C})\) be the differentiation operator, that is, \(T(f) := f'\). In \cite[Example 2.4]{Bayart06}, it is shown that \(T\) satisfies the Frequently Hypercyclicity Criterion. By Corollary \ref{fhc-irre}, for every \((\mu_m)_m \in \Upsilon\), \(T\) admits both a dense absolutely \((\mu_m)\)-irregular manifold and a dense \((\mu_m)\)-distributionally irregular manifold.
\end{example}

\begin{proposition}
	Let \(\mathrm{H}(\mathbb{D})\) denote the space of holomorphic functions on \(\mathbb{D}\), endowed with the topology of uniform convergence on compact sets. Let \(\psi\) be an analytic self-map of \(\mathbb{D}\). Then \(\psi\) has no fixed point in \(\mathbb{D}\) if and only if the composition operator
	\[
	C_{\psi}:\mathrm{H}(\mathbb{D})\to \mathrm{H}(\mathbb{D})
	\]
	admits a dense absolutely \((\mu_m)\)-irregular manifold and a dense distributionally \((\mu_m)\)-irregular manifold for every \((\mu_m)_m\in\Upsilon\).
\end{proposition}
\begin{proof}
	We follow the argument of \cite[Theorem 34]{BeBoMuPe13}. Assume that \(\psi\) has no fixed point in \(\mathbb{D}\). By the Denjoy--Wolff Theorem, there exists \(p\in \partial\mathbb{D}\) such that \((\psi^{n})\) converges uniformly on compact sets to \(p\).
	
	Let \(X_{0}\) denote the set of all functions that are continuous on \(\overline{\mathbb{D}}\), analytic on \(\mathbb{D}\), and vanish at \(p\). Then \(X_{0}\) is dense in \(\mathrm{H}(\mathbb{D})\) and
	\[
	\lim_{n\to\infty} C_{\psi}^{n}f=0
	\]
	for every \(f\in X_{0}\).
	
	Fix \((\mu_{m})_{m}\in \Upsilon\). For each \(k\in\mathbb{N}\), define
	\[
	g_{k}(z)=\frac{1}{k(p-z)}.
	\]
	Then \(g_{k}\in \mathrm{H}(\mathbb{D})\) and \(g_{k}\to 0\) in \(\mathrm{H}(\mathbb{D})\).
	
	There exist \(\beta\in\mathbb{N}\) and \(n_{k}\in\mathbb{N}\) such that
	\[
	\|C_{\psi}^{j}g_{k}\|_{\beta}>\frac12,
	\qquad j>n_{k}.
	\]
	We can choose a strictly increasing sequence \((m_{k})_{k}\) of positive integers such that
	\[
	\lim_{k\to\infty}
	\mu_{m_{k}}
	\Bigl(
	\bigl\{
	j\in\mathbb{N}:
	\|C_{\psi}^{j}g_{k}\|_{\beta}>\tfrac12
	\bigr\}
	\Bigr)
	=1.
	\]
	
	By Proposition~\ref{dist-unbounded}, \(C_{\psi}\) admits a \(\beta\)-distributionally \((\mu_m)\)-unbounded vector. Therefore, Theorem~\ref{dense-disti} implies that \(C_{\psi}\) admits a dense distributionally \((\mu_m)\)-irregular manifold.
	
	Since \((\mu_m)_m\in\Upsilon\) was arbitrary, Theorem~\ref{TOS} yields that \(C_{\psi}\) admits both a dense absolutely \((\mu_m)\)-irregular manifold and a dense distributionally \((\mu_m)\)-irregular manifold for every \((\mu_m)_m\in\Upsilon\).
	
	Conversely, consider the observation scheme
	\[
	\left(
	\mu_{m}:=\frac1m\sum_{k=1}^{m}\delta_{k}
	\right)_{m}\in\Upsilon.
	\]
	By assumption, \(C_{\psi}\) admits a dense distributionally \((\mu_m)\)-irregular manifold and hence is distributionally chaotic. It follows from \cite[Theorem 34]{BeBoMuPe13} that \(\psi\) has no fixed point in \(\mathbb{D}\).
\end{proof}

\begin{example}\label{ejemplo}
	Let \(T\) be the unilateral weighted backward shift on \(\ell^{p}(\mathbb{N})\) with
	\(1 \leq p < \infty\), defined by \(T e_{1} := 0\) and \(T e_{k} := w_{k}e_{k-1}\) for \(k > 1\). If \(w_{k} := (\frac{k}{k-1})^{\alpha}\) with \(0 < \alpha < \frac{1}{p}\), then \(T\) is mixing and absolutely Cesàro bounded \cite[Theorem 2.1]{BeBoMuPe20}. By Theorem \ref{TOS}, there exists a residual subset \(\Upsilon_0 \subset \Upsilon\) such that
	for every \((\mu_m)_m \in \Upsilon_0\), \(T\) admits a dense absolutely \((\mu_m)\)-irregular manifold, while \(T\) admits no \((\mu_m)\)-distributionally irregular vector.
\end{example}

\begin{theorem}
	Let \(X\) be a separable real or complex Banach space and let 
	\(T\in \mathcal{L}(X)\) be an operator which is not power bounded. 
	Assume that
	\begin{itemize}
		\item the set \(\{x\in X : T^{n}x \to 0\}\) is dense in \(X\), and
		\item there exists a decreasing sequence \((E_n)_{n\in\mathbb{N}}\) of infinite-dimensional closed subspaces of \(X\) such that
		\[
		\sup_{n\in\mathbb N}\|T^{n}|_{E_{n}}\|<\infty.
		\]
	\end{itemize}
	Then there exists a residual subset \(\Upsilon_{0}\subset \Upsilon\) such that, for every \((\mu_{m})\in \Upsilon_{0}\), the operator \(T\) admits both
	\begin{itemize}
		\item a dense absolutely \((\mu_{m})\)-irregular manifold, and
		\item a closed infinite-dimensional absolutely \((\mu_{m})\)-irregular manifold.
	\end{itemize}
\end{theorem}

\begin{proof}
	Let \(\Upsilon_{0}\subset \Upsilon\) be the residual subset provided by Theorem~\ref{gip}. Then, for every \((\mu_{m})\in \Upsilon_{0}\), the operator \(T\) admits a dense absolutely \((\mu_{m})\)-irregular manifold.
	
	On the other hand, by Theorem~\ref{condi-suf-space-irre}, for every \((\mu_{m})\in \Upsilon_{0}\), the operator \(T\) also admits a closed infinite-dimensional absolutely \((\mu_{m})\)-irregular manifold. This completes the proof.
\end{proof}

\begin{theorem}\label{gen-(mu_i)}
	Let \(X\) be an infinite-dimensional separable Banach space. Then the set of operators \(T \in \mathcal{L}(X)\) such that for each \((\mu_m)_m \in \Upsilon\), \(T\) admits a dense absolutely \((\mu_m)_m\)-irregular manifold, is SOT-dense in \(\mathcal{L}(X)\).
\end{theorem}

In \cite{Shk08}, Shkarin proved that on every separable infinite-dimensional Banach space \(X\), there exists a bounded operator \(T \in \mathcal{L}(X)\) satisfying the Kitai criterion and admitting a nonzero fixed point (see \cite[p.~1668]{Shk08}). Recall that an operator \(T \in \mathcal{L}(X)\) satisfies the \emph{Kitai criterion} if there exist two dense subsets \(E, F \subset X\) and a map \(S \colon F \to F\) such that \(TSy = y\), \(S^k y \to 0\), and \(T^k x \to 0\) as \(k \to \infty\) for all \(x \in E\) and \(y \in F\).

\begin{proof}[Proof of Theorem \ref{gen-(mu_i)}]
	Fix a bounded operator \(T\) on \(X\) satisfying the Kitai criterion and admitting a nonzero vector \(q \in X\) such that \(Tq = q\). Let \((\mu_m)_{m \in \mathbb{N}} \in \Upsilon\). By the Kitai criterion, the set
	\[
	Z_0 := \left\{ x \in X : \lim_{n \to \infty} T^n x = 0 \right\}
	\]
	is dense in \(X\). Since \((\mu_m)_m \in \Upsilon\), it follows that
	\[
	Z_0 \subset \left\{ x \in X : \lim_{m \to \infty} \int \|T^{t}x\|\,d\mu_m(t) = 0 \right\}.
	\]
	Consequently, by Theorem~\ref{dense-lineable}, the operator \(T\) admits a dense \((\mu_m)_m\)-irregular manifold \(E \subset X\). Indeed, for every \(m \in \mathbb{N}\),
	\[
	\int \|T^{t}q\|\,d\mu_m(t) = \|q\|.
	\]
	Consider the set
	\[
	D := \{ A^{-1}TA : A \in \mathcal{L}(X)\ \text{invertible} \},
	\]
	which is SOT-dense in \(\mathcal{L}(X)\). Finally, for every operator
	\[
	S := A^{-1}TA \in D,
	\]
	it is straightforward to verify that \(A^{-1}(E)\) is a dense absolutely \((\mu_m)_m\)-irregular manifold for \(S\). This completes the proof.
\end{proof}

\begin{question}
	Let \(X\) be an infinite-dimensional separable Banach space. Is the set of operators \(T \in \mathcal{L}(X)\) such that, for every \((\mu_m)_m \in \Upsilon\), \(T\) admits a dense distributionally \((\mu_m)\)-irregular manifold, dense in \(\mathcal{L}(X)\) with respect to the strong operator topology?
\end{question}

Item~(III) in Theorem~\ref{TOS} together with Example~\ref{ejemplo} establishes a distinction between two types of chaotic behavior on a generic subset of observation schemes. Motivated further by \cite[Theorem 25]{BeBoPeWu18}, we propose the following question.

\begin{question}
	Does there exist a bounded operator \(T\) on a separable Banach space \(X\) and a dense subset \(\Upsilon_{0}\subset \Upsilon\) such that, for every \((\mu_{m})_{m}\in \Upsilon_{0}\), the operator \(T\) is densely \((\mu_{m})\)-distributionally chaotic but not densely \((\mu_{m})\)-Li--Yorke chaotic?
\end{question}


\bibliographystyle{abbrv}
\bibliography{md_references.bib}

\end{document}